\documentclass{amsart}

\usepackage{amsmath,amssymb,amsthm}
\usepackage[shortlabels]{enumitem}
\usepackage{xcolor}
\usepackage{hyperref}

\hypersetup{
  colorlinks=true,
  linkcolor=blue!55!black,
  citecolor=blue!55!black,
  urlcolor=blue!65!black
}

\usepackage[initials, backrefs]{amsrefs}

\usepackage[capitalise,noabbrev]{cleveref}

\theoremstyle{plain}
\newtheorem{theorem}{Theorem}[section]
\newtheorem{prop}[theorem]{Proposition}
\newtheorem{lem}[theorem]{Lemma}
\newtheorem{cor}[theorem]{Corollary}

\theoremstyle{definition}
\newtheorem{definition}[theorem]{Definition}
\newtheorem{exa}[theorem]{Example}
\newtheorem{algorithm}[theorem]{Algorithm}

\newtheorem{rem}[theorem]{Remark}

\newcommand{\Z}{\mathbb Z}

\newcommand{\Q}{\mathbb Q}
\newcommand{\C}{\mathbb C}

\newcommand{\PP}{\mathbb P}

\newcommand\cC{\mathcal C}
\newcommand\cM{\mathcal M}
\newcommand\cS{\mathcal S}

\DeclareMathOperator{\Aut}{Aut}
\DeclareMathOperator{\disc}{disc}
\DeclareMathOperator{\Div}{div}
\DeclareMathOperator{\Gal}{Gal}
\DeclareMathOperator{\gon}{gon}
\DeclareMathOperator{\Lev}{Lev}
\DeclareMathOperator{\ord}{ord}

\crefname{theorem}{Thm.}{theorems}
\crefname{prop}{Prop.}{propositions}
\crefname{lem}{Lem.}{lemmas}
\crefname{cor}{Cor.}{corollaries}
\crefname{definition}{Def.}{definitions}
\crefname{exa}{Exa.}{examples}
\crefname{rem}{Rem.}{remarks}
\crefname{algorithm}{Alg.}{algorithms}
\crefname{equation}{Eq.}{}

\title[Deciding superellipticity]{Deciding superellipticity and computing the Weierstrass normal form}

\author{T. Shaska}
\address{Department of Computer Science and Engineering, Oakland University, Rochester, MI 48309, USA}
\email{shaska@oakland.edu}



\begin{document}

\begin{abstract}
Let \( \cS_{g,n} \subset \cM_g \) be the locus of curves of genus \( g \geq 2 \) admitting a model \( y^n = h(x) \) with \( h \) separable;
such  curves $C$  have a cyclic group \( C_n \leq \Aut(C) \) of order \( n \) with \( C/C_n \cong \PP^1 \).
We give an algorithm which, given an absolutely irreducible plane model \( F(x,y) = 0 \) of a curve \( C \) over a field \( k_0 \) of characteristic zero, decides for which \( n \) the curve lies in \( \cS_{g,n} \) and returns a model \( y^n = h(x) \) together with the birational transformation to it. 
\end{abstract}

\maketitle

\setcounter{tocdepth}{1}

\section{Introduction}\label{sec:intro}

Let $k$ be an algebraically closed field of characteristic zero and let $\cM_g$ be the moduli space of curves of genus $g\ge 2$ over $k$. A curve $C$ admits a cyclic cover of degree $n$ when some $\tau\in\Aut(C)$ of order $n$ has quotient $C/\langle\tau\rangle$ of genus zero; every such curve has an affine model $y^n=h(x)$ with $h\in k[x]$. When $h$ can be taken separable, $C$ is called superelliptic of level $n$; the term is used with different meanings in the literature, and our convention is fixed in \cref{sec:cyclic}. We write $\cS_{g,n}\subset\cM_g$ for the locus of superelliptic curves of level $n$.

The loci $\cS_{g,n}$ are the part of $\cM_g$ where the theory of automorphism groups is explicit. For a curve in $\cS_{g,n}$ whose cyclic group of level $n$ is normal in $\Aut(C)$, the tables of \cite{beshaj2011, sanjeewa2008, malmendier2019} give the full automorphism group and a parametric equation in terms of the level and the signature of the cover; no comparable description exists for curves outside these loci.
Moreover, of the $41$ loci of curves with prescribed full automorphism group in genus four, $28$ consist of superelliptic curves, and \cite{malmendier2019} poses the corresponding ratio in higher genus as an open problem. The basic question in applying the theory is therefore the following. Given a curve $C$ by an arbitrary absolutely irreducible plane model $F(x,y)=0$, decide for which $n$ it lies in $\cS_{g,n}$, and compute a birational transformation to a model $y^n=h(x)$. For the hyperelliptic level $n=2$ this problem was solved by van Hoeij \cite{vanhoeij2002}; for $n\ge 3$ no such algorithm is known.

The hyperelliptic case shows why the general case needs a new idea. Van Hoeij decides whether $C$ is hyperelliptic from the products of its holomorphic differentials, computes a generator $x_0$ of the unique subfield of index two, and reads the normal form off the minimal polynomial of a second function over $k(x_0)$. The degree-two map is a ratio of holomorphic differentials, so it exists on every hyperelliptic curve and no choice of a point is needed. For $n\ge 3$ the degree-$n$ map is not read off the canonical map: a cyclic function of degree $n$ has a single pole of order $n$ at a totally ramified point $P$, so $h^0(nP)\ge 2$, and $n$ must be a non-gap at $P$, which holds only at special points. The general algorithm of Hess \cite{hess2004} for isomorphisms of function fields also proceeds through Weierstrass points, but it computes automorphisms over the given constant field, not over its algebraic closure, and it returns neither the cyclic subgroups nor a normal form. Its computation of Weierstrass points \cite{hess2002wp} and Riemann--Roch spaces \cite{hess2002} is used here as a subroutine.

A second approach would be through subfields. Whether $C$ admits a cyclic cover of degree $n$ is the question of whether $k(C)$ has a subfield of index $n$ and genus zero over which $k(C)$ is Galois, and algorithms for the subfields of a function field are available \cite{vanhoeij2013}; this is how the hyperelliptic case is posed in \cite{vanhoeij2002}. It does not settle the present question. The subfields are computed over $k_0$, while membership in $\cS_{g,n}$ is a property of $C_{\bar k_0}$; a genus zero subfield is a conic, which is $\PP^1_{k_0}$ only when it has a rational point, as in \cref{rem:descent}; and a cyclic subfield of index $n$ does not by itself give a separable model, which by \cref{thm:sep} need not exist at all, as in \cref{ex:nonsep}. The number of subfields is moreover not bounded by a polynomial in $n$, which is the motivation for the generating subfields of \cite{vanhoeij2013}, whereas the search below runs over the at most $g^3-g$ Weierstrass points, filtered by weight.

The identification of the special points rests on the following statement, proved in \cref{sec:wp}. Let $C: y^n=h(x)$ with $h\in k[x]$ separable of degree $m$, let $g\ge 2$, and let $P$ be a point of $C$ above a root of $h$. If $m\ge 3$, then $n\le g+2$ and $P$ is a Weierstrass point of $C$. If $m=2$, then $C$ is hyperelliptic, $n\in\{2g+1,2g+2\}$, and the vanishing sequence at $P$ is the generic one. The degree two map of $C$ is then $y$ and not $x$, and its branch points do not lie above the roots of $h$.

This makes the search finite. In characteristic zero a curve of genus $g$ has at most $g^3-g$ Weierstrass points, so for each level $n\le g+2$ the branch divisor of a separable model is one of finitely many sums of Weierstrass points, selected by their vanishing sequences, which are known at branch points of cyclic covers \cite{towse1996,   shor}. The search is over these divisors. Each candidate is decided exactly by a Kummer criterion (\cref{prop:cert}), and each output is verified by polynomial division. This is \cref{alg:main}, and \cref{thm:complete} states that it decides membership in $\cS_{g,n}$: run over the field $k_n$ generated by the Weierstrass points of admissible weight and by $\zeta_n$, it returns exactly the levels $n$ with $[C]\in\cS_{g,n}$, with a model over $k_n$. The question over the field of definition $k_0$ is arithmetic and is treated in \cref{sec:arith}: a level realised by a cover defined over $k_0$ is found over $k_0$, with at most one extension of degree at most $\deg h$ (\cref{cor:complete-k0}), and the candidates are $k_0$-rational sums of closed points of the Weierstrass divisor. A level realised only by cyclic covers that admit no separable model lies outside that statement. \Cref{rem:incomplete} treats such levels: the reduction to Weierstrass points still recovers those realised by a cover with two totally ramified points and $n\le g$, and it gives branch data for which the reduction fails.

The normal form obtained has $h\in k[u]$, but $h$ need not be separable, and no construction can make it so: by \cref{thm:sep} a cyclic morphism with branch data $(l_1,\dots,l_r)$ admits a model $v^n=h(u)$ with $h$ separable if and only if some unit $c$ modulo $n$ satisfies $c\,l_j\equiv 1\pmod n$ for all $j$ with at most one exception.

The paper is organised as follows. \Cref{sec:cyclic} recalls cyclic covers, branch data and the level spectrum, separates the level from the gonality, records the vanishing sequence at a totally ramified point, and describes the loci $\cS_{g,n}$ as Hurwitz loci in $\cM_g$. \Cref{sec:wp} proves the statement on Weierstrass points and the bounds on $n$. \Cref{sec:wnf} constructs the normal form and characterises the separable case. \Cref{sec:alg} states the algorithm and proves that it decides membership in $\cS_{g,n}$. \Cref{sec:arith} descends the algorithm to the field of definition and reports timings for the implementation.

\section{Cyclic covers and superelliptic curves}
\label{sec:cyclic}

This section fixes the conventions, defines the level and the level spectrum, records the vanishing sequence of the holomorphic differentials at a totally ramified point, and describes the loci $\cS_{g,n}$ as Hurwitz loci in $\cM_g$.

Throughout, $k$ is an algebraically closed field of characteristic zero, $C$ is a smooth projective curve over $k$ of genus $g\ge 2$, and $n\ge 2$ is an integer; $k$ contains a primitive $n$-th root of unity $\zeta_n$. For a divisor $D$ on $C$ we write $L(D)$ for its Riemann--Roch space and $h^0(D)=\dim L(D)$.

The vanishing sequence of $H^0(C,\Omega^1)$ at a point $P$ is the increasing sequence of the $g$ orders $\ord_P(\omega)$ taken by non-zero holomorphic differentials. It is the same for all but finitely many points; that common sequence is the generic sequence of $C$, and $P$ is a \textbf{Weierstrass point} when its sequence differs from it. 
In characteristic zero the generic sequence is $(0,1,\dots,g-1)$ and there are at most $g^3-g$ Weierstrass points.

\subsection{Branch data}\label{sec:branch}

Let $\tau\in\Aut(C)$ have order $n\ge 2$ with $C/\langle\tau\rangle$ of genus zero. Then $C$ has an affine model
\begin{equation}\label{eq:model}
y^n=\prod_{j=1}^{s}(x-p_j)^{l_j},
\end{equation}
with $p_1,\dots,p_s\in k$ pairwise distinct, $l_j\in\{1,\dots,n-1\}$ and $\gcd(n,l_1,\dots,l_s)=1$. In this model $\tau(x,y)=(x,\zeta_n y)$ and the quotient map is $\pi(x,y)=x$. Put $m=l_1+\dots+l_s$. The point $x=\infty$ is a branch point of $\pi$ if and only if $n\nmid m$, and then its local monodromy is $l_\infty\equiv -m\pmod n$. The tuple $(l_1,\dots,l_s)$, together with $l_\infty$ when $n\nmid m$, is the branch data of the pair $(C,\langle\tau\rangle)$.

Write $d_j=\gcd(n,l_j)$ and $e_j=n/d_j$. The fibre of $\pi$ over $p_j$ consists of $d_j$ points, each of ramification index $e_j$. The point $p_j$ is totally ramified when $d_j=1$. Similarly $d_\infty=\gcd(n,m)$ and $e_\infty=n/d_\infty$. The Riemann--Hurwitz formula reads
\begin{equation}\label{eq:RH}
2g-2=-2n+\sum_{j}\bigl(n-\gcd(n,l_j)\bigr),
\end{equation}
the sum being over all branch points including $\infty$. When every $l_j=1$ and $h(x)=\prod_j(x-p_j)$ has degree $m$, this reads
\begin{equation}\label{eq:genus}
2g=(n-1)(m-1)-\gcd(n,m)+1.
\end{equation}
For our purposes, we will use the definition below. 

\begin{definition}\label{def:superelliptic}
A curve $C$ of genus $g\ge 2$ is \textbf{superelliptic of level $n$} if it admits an affine model $y^n=h(x)$ with $h\in k[x]$ separable, that is a model \cref{eq:model} in which every $l_j=1$.
\end{definition}

The term is used with different meanings in the literature. In \cite{beshaj2011, malmendier2019} a curve is superelliptic when a cyclic subgroup $H\le\Aut(C)$ with $C/H$ of genus zero is normal in $\Aut(C)$, with no condition on the model; in \cite{hidalgo, hidalgo2024} the generator is required to be central in $\Aut(C)$, and central in its normalizer for the generalized case. In \cite{Obus2021} the model $y^n=h(x)$ is required to be separable with $n\mid\deg h$ or $\gcd(n,\deg h)=1$, so that every branch point is totally ramified, and $\langle\tau\rangle$ is required to be normal in $\Aut(C)$; dropping the normality gives the pre-superelliptic curves of \cite{Obus2021}.

These classes differ from \cref{def:superelliptic} in both directions. A separable model with $1<\gcd(n,\deg h)<n$ satisfies \cref{def:superelliptic} and fails the ramification condition of \cite{Obus2021} in every separable presentation, since a M\"obius transformation moving a totally ramified branch point to infinity destroys separability. Conversely, a cyclic cover satisfying the group conditions need not admit a separable model at all: \cref{thm:sep} determines when it does, and \cref{ex:nonsep} gives a curve with a cyclic cover of degree four and no separable model of level four. \Cref{def:superelliptic} imposes no condition on $\Aut(C)$: the algorithm decides the existence of the separable model, and the group conditions are properties of $\Aut(C)$ that can be tested afterwards. The standing model \cref{eq:model} does not assume separability; \cref{thm:sep} determines when it can be arranged.

\begin{lem}\label{lem:equiv}
Two branch data define isomorphic pairs $(C,\langle\tau\rangle)$ if and only if they agree after a permutation of the indices, a M\"obius transformation of the points $p_j$, and multiplication of all $l_j$ by one unit $c\in(\Z/n\Z)^\times$.
\end{lem}

\begin{proof}
An isomorphism $\varphi: C\to C'$ with $\varphi\langle\tau\rangle\varphi^{-1}=\langle\tau'\rangle$ satisfies $\varphi\tau\varphi^{-1}=\tau'^{\,c}$ for a unit $c$, and descends to an isomorphism $\bar\varphi$ of the quotients $\PP^1_x\to\PP^1_{x'}$, a M\"obius transformation carrying branch points to branch points. The local monodromy at $\bar\varphi(p_j)$ with respect to $\tau'$ is that at $p_j$ with respect to $\tau$ multiplied by $c$. Conversely, if two data agree after these operations, then after the M\"obius transformation and the change of generator the right hand sides of the two models \cref{eq:model} have the same order modulo $n$ at every place of $k(x)$, so they define the same class in $k(x)^\times/(k(x)^\times)^n$, and by Kummer theory the two function fields coincide as extensions of $k(x)$ with $\tau$ corresponding to $\tau'$. See also \cite[\S 4]{hidalgo2024}.
\end{proof}

The unit $c$ records the choice of the generator among $\tau,\tau^2,\dots$, and the M\"obius transformation records which branch point is placed at infinity.


A curve may admit cyclic covers of more than one degree.

\begin{definition}\label{def:lev}
The \textbf{level spectrum of $C$} is the set $\Lev  (C)$ of integers $n>1$ such that $C$ admits a cyclic cover $C\to\PP^1$ of degree $n$.
\end{definition}

By Galois theory $n\in\Lev(C)$ if and only if $\Aut(C)$ contains a cyclic subgroup of order $n$ whose quotient has genus zero. Since $\Aut(C)$ is finite, $\Lev(C)$ is a finite set. It is an invariant of $C$ and not of a chosen model.

\begin{exa}\label{ex:picard}
Let $C: y^3=x^4-1$. Then $g=3$ by \cref{eq:genus}. Rearranging gives $x^4=y^3+1$, so $x$ realises $C$ as a cyclic cover of degree four with model $x^4=y^3+1$, and $\{3,4\}\subseteq\Lev(C)$. Both models are of the form $v^n=h(u)$ with $h$ separable, with $(n,m)=(3,4)$ and $(4,3)$.
\end{exa}

 The gonality $\gon(C)$ is the least degree of a non-constant morphism $C\to\PP^1$, equivalently the least $d$ for which $C$ carries a pencil $g^1_d$. It does not bound the levels. The gonality of a curve of genus $g$ is at most $\lfloor(g+3)/2\rfloor$, while by \cref{thm:bound} the level of a separable model may be as large as $g+2$, which is larger for every $g\ge 2$.
 
\subsection{Differentials at a totally ramified point}\label{sec:diff}

The gap sequence at a totally ramified place of a Kummer extension is determined in \cite[Theorem 3.5, Corollary 3.6]{abdon2019}, and for separable $h$ in \cite{towse1996} and \cite[Theorem 6.1]{abdon2019}. We record the vanishing sequence of the holomorphic differentials, which is the same information, in the form used by the algorithm.

Let $C$ be given by \cref{eq:model}. Fix an index $i$ with $d_i=1$ and let $P$ be the unique point of $C$ above $p_i$. For $a\ge 0$ and $0\le b<n$ put
\begin{equation}\label{eq:omega}
\omega_{a,b}=(x-p_i)^a\prod_{j=1}^{s}(x-p_j)^{\lfloor b\,l_j/n\rfloor}\,\frac{dx}{y^b}.
\end{equation}

\begin{prop}\label{prop:basis}
The differential $\omega_{a,b}$ is holomorphic on $C$ if and only if $(a,b)\in A$, where
\begin{equation}\label{eq:Aset}
A=\Bigl\{(a,b): a\ge 0,\ 0\le b<n,\ a+\sum_{j}\Bigl\lfloor\frac{b\,l_j}{n}\Bigr\rfloor\le\frac{bm}{n}-1-\frac{1}{e_\infty}\Bigr\}.
\end{equation}
The set $\{\omega_{a,b}\}_{(a,b)\in A}$ is a basis of $H^0(C,\Omega^1)$, and
\begin{equation}\label{eq:orders}
\ord_P(\omega_{a,b})=na+n-1-\bigl(b\,l_i \bmod n\bigr),
\end{equation}
where $b\,l_i\bmod n$ denotes the least non-negative residue.
\end{prop}

\begin{proof}
At a point over $p_j$ we have $\ord(x-p_j)=e_j$ and $\ord(y)=l_j/d_j$; the ramification is tame, so $\ord(dx)=e_j-1$ and
\[
\begin{split}
\ord(\omega_{a,b})&=e_j\Bigl\lfloor\frac{b\,l_j}{n}\Bigr\rfloor+e_j-1-\frac{b\,l_j}{d_j} =e_j-1-\frac{b\,l_j\bmod n}{d_j},
\end{split}
\]
since $b\,l_j/d_j=e_j\cdot b\,l_j/n$ and the term $(x-p_i)^a$ contributes only when $j=i$. As $d_j$ divides both $n$ and $l_j$, the residue $b\,l_j\bmod n$ is a multiple of $d_j$ not exceeding $n-d_j$, so the order is a non-negative integer. Thus $\omega_{a,b}$ is holomorphic at every finite branch point, and at $P$ the factor $(x-p_i)^a$ adds $na$, which gives \cref{eq:orders}. At a point over $x=\infty$ we have $\ord(x)=-e_\infty$ and $\ord(y)=-m/d_\infty$, so with $A_{a,b}=a+\sum_j\lfloor b\,l_j/n\rfloor$ the order of $\omega_{a,b}$ is $-e_\infty A_{a,b}-e_\infty-1+bm/d_\infty$. Non-negativity of this order is the inequality in \cref{eq:Aset}. Away from the branch points $\omega_{a,b}$ is regular. Since $\gcd(n,l_i)=1$, the map $b\mapsto b\,l_i\bmod n$ is a bijection of $\{0,\dots,n-1\}$, so the orders \cref{eq:orders} are pairwise distinct and the $\omega_{a,b}$ with $(a,b)\in A$ are linearly independent. Conversely, $H^0(C,\Omega^1)$ is $\tau$-stable, so every holomorphic differential is the sum of its eigencomponents $\varphi(x)\,dx/y^b$, each holomorphic. Holomorphy of such a component at the finite points forces $\varphi=\prod_j(x-p_j)^{\lfloor b\,l_j/n\rfloor}\psi$ with $\psi\in k[x]$, by the orders computed above, and holomorphy over $x=\infty$ bounds $\deg\psi$ by the inequality in \cref{eq:Aset}. Hence the $\omega_{a,b}$ with $(a,b)\in A$ span $H^0(C,\Omega^1)$, and $|A|=g$.
\end{proof}

\begin{cor}\label{prop:filter}
The vanishing sequence of $H^0(C,\Omega^1)$ at $P$ is
\begin{equation}\label{eq:sn}
s=\bigl(\,na+n-1-(b\,l_i\bmod n) : (a,b)\in A\,\bigr),
\end{equation}
arranged in increasing order.
\end{cor}

\begin{proof}
A set of $g$ holomorphic differentials with pairwise distinct orders at $P$ realises the vanishing sequence at $P$.
\end{proof}

When every $l_j=1$ and $m=\deg h$, substituting $b'=n-1-b$ turns \cref{eq:Aset} into $\{(a,b'): a\ge 0,\ 0\le b'<n,\ an+b'm\le 2g-2\}$ and \cref{eq:orders} into $\ord_P=an+b'$. This is the basis of \cite[Proposition 13]{shor}. We write $s_{n,m}$ for the sequence \cref{eq:sn} in that case, so that
\begin{equation}\label{eq:snm}
s_{n,m}=\bigl(\,an+b : a\ge 0,\ 0\le b<n,\ an+bm\le 2g-2\,\bigr).
\end{equation}

\begin{exa}\label{ex:seq}
For $(n,m)=(4,5)$ formula \cref{eq:genus} gives $g=6$ and $s_{4,5}=(0,1,2,4,5,8)$. For $(n,m)=(5,6)$ one gets $g=10$ and $s_{5,6}=(0,1,2,3,5,6,7,10,11,15)$. Neither is of the form $(0,1,\dots,q-1,q,q+n,q+2n,\dots)$. That shape occurs only for $n=2$.
\end{exa}

\begin{exa}\label{ex:top}
Let $C: y^n=x^3-x$ with $n\ge 4$, so $m=3$ and $\gcd(n,3)\in\{1,3\}$. If $3\nmid n$, then $g=n-1$ and the pairs in \cref{eq:snm} are $(0,b)$ with $3b\le 2n-4$ and $(1,b)$ with $3b\le n-4$, so
\[
s_{n,3}=\bigl(0,1,\dots,\lfloor (2n-4)/3\rfloor,\ n,n+1,\dots,n+\lfloor (n-4)/3\rfloor\bigr).
\]
If $3\mid n$, then $g=n-2$ and
\[
s_{n,3}=\bigl(0,1,\dots,\tfrac{2n}{3}-2,\ n,n+1,\dots,\tfrac{4n}{3}-2\bigr).
\]
For $n=5$ this is $(0,1,2,5)$, for $n=6$ it is $(0,1,2,6)$, and for $n=9$ it is $(0,1,2,3,4,9,10)$. In every case the sequence contains $1$ and contains $n>g-1$, so it is neither the generic sequence $(0,1,\dots,g-1)$ nor the sequence $(0,2,\dots,2g-2)$ of a hyperelliptic Weierstrass point.
\end{exa}

\subsection{Hurwitz loci}\label{sec:hurwitz}

Let $\cM_g$ be the moduli space of smooth projective curves of genus $g\ge 2$ over $k$. We describe the loci of superelliptic curves in $\cM_g$ in the language of \cite{mssv2001, malmendier2019}.

Let $G$ be a finite group and let $\mathbf C=(C_1,\dots,C_r)$ be a tuple of non-trivial conjugacy classes of $G$. A curve $X$ of genus $g$ with a faithful action of $G$ is of ramification type $(g,G,\mathbf C)$ if the branch points of $X\to X/G$ can be labelled $p_1,\dots,p_r$ so that $C_i$ is the class of a distinguished inertia generator over $p_i$. Let $\cM(g,G,\mathbf C)\subset\cM_g$ be the locus of curves admitting such an action. By \cite[Lemma 16]{malmendier2019}, which goes back to \cite{mssv2001}, every component of $\cM(g,G,\mathbf C)$ has dimension
\begin{equation}\label{eq:delta}
\delta(g,G,\mathbf C)=3g_0-3+r,
\end{equation}
where $g_0$ is the genus of $X/G$. If $H\le G$, restriction of the action gives
\begin{equation}\label{eq:incl}
\cM(g,G,\mathbf C)\subset\cM(g,H,\mathbf\Delta)
\end{equation}
for the induced type $\mathbf\Delta$ \cite[\S 6.5]{malmendier2019}. The inclusions among the loci with $G$ the full automorphism group were determined in \cite{mssv2001} for small genus; the poset for $g=3$ and the $41$ loci for $g=4$ are displayed in \cite[\S 6.5]{malmendier2019}.

\subsection{The loci $\cS_{g,n}$}\label{sec:Sgn}

For a cyclic group $G=\langle\tau\rangle$ of order $n$ acting with quotient $\PP^1$, the ramification type is the branch data of \cref{sec:branch}, and \cref{eq:delta} reads $\delta=r-3$ with $r$ the number of branch points. Let $\mathbf C_m$ be the branch data $(1,\dots,1)$ of length $m$, together with $l_\infty\equiv -m$ when $n\nmid m$; this is the branch data of a cyclic cover of degree $n$ with a separable model $y^n=h(x)$, $\deg h=m$. The genus of such a cover is given by \cref{eq:genus}, and the datum is determined by $g$ and $n$ alone.

\begin{lem}\label{lem:unique-m}
Let $g\ge 2$ and $n\ge 2$. Then \cref{eq:genus} has at most two solutions $m\ge 2$. If it has two, they are $m$ and $m+1$ with $n\mid m+1$, and then $\mathbf C_m$ and $\mathbf C_{m+1}$ agree up to \cref{lem:equiv}.
\end{lem}

\begin{proof}
Let $m<m'$ both satisfy \cref{eq:genus} and put $d=\gcd(n,m)$ and $d'=\gcd(n,m')$. Subtracting the two instances of \cref{eq:genus} gives $(n-1)(m'-m)=d'-d$. Since $1\le d,d'\le n$, the right hand side is at most $n-1$, so $m'=m+1$ and $d'-d=n-1$, which forces $d=1$ and $d'=n$, that is $n\mid m+1$. Then $-m\equiv 1\pmod n$, so $\mathbf C_m$ is $(1,\dots,1)$ of length $m+1$, the last entry being the local monodromy at infinity; and $n\mid m+1$ makes infinity unramified for $\mathbf C_{m+1}$, which is $(1,\dots,1)$ of length $m+1$ as well. The two agree after the M\"obius transformation of \cref{lem:equiv} carrying the branch point at infinity to a finite point.
\end{proof}

When \cref{eq:genus} has a solution we write $m(g,n)$ for the smaller one, and $\mathbf C_{g,n}$ for the branch datum $\mathbf C_{m(g,n)}$, which by \cref{lem:unique-m} is the only one attached to the pair $(g,n)$.

\begin{definition}\label{def:Sgn}
Let $g\ge 2$ and $n\ge 2$. The superelliptic locus of level $n$ is
\[
\cS_{g,n}=\{\,[C]\in\cM_g : C \text{ admits a model } y^n=h(x) \text{ with } h\in k[x] \text{ separable}\,\}.
\]
\end{definition}

\begin{prop}\label{prop:Sgn-irred}
Let $g\ge 2$ and $n\ge 2$ with $\cS_{g,n}\ne\emptyset$, and put $m=m(g,n)$. Then
\begin{equation}\label{eq:Sunion}
\cS_{g,n}=\cM(g,C_n,\mathbf C_{g,n}),
\end{equation}
and $\cS_{g,n}$ is irreducible of dimension
\begin{equation}\label{eq:dimS}
\dim\cS_{g,n}=
\begin{cases}
m-3, & n\mid m,\\
m-2, & n\nmid m.
\end{cases}
\end{equation}
\end{prop}

\begin{proof}
A curve in $\cS_{g,n}$ has a separable model $y^n=h(x)$ whose degree solves \cref{eq:genus}, and by \cref{lem:equiv} and \cref{thm:sep} the branch data of the corresponding cover is $\mathbf C_{\deg h}$. If $\deg h=m+1$, then moving a root of $h$ to infinity replaces $\mathbf C_{m+1}$ by $\mathbf C_m$, by \cref{lem:unique-m}, so $C$ has a separable model of degree $m$. Conversely a curve of ramification type $(g,C_n,\mathbf C_{g,n})$ has all but at most one entry of its branch data equal to $1$, hence a separable model by \cref{thm:sep}. This is \cref{eq:Sunion}.

Scaling $y$ makes $h$ monic, so $\cS_{g,n}$ is the set of classes of the curves
\[
y^n=\prod_{j=1}^{m}(x-p_j),\qquad p_1,\dots,p_m\in k \text{ distinct}.
\]
Let $U\subset\mathbb A^m$ be the complement of the discriminant locus, a non-empty open subvariety of $\mathbb A^m$ and therefore irreducible. The displayed equations define a family of smooth projective curves of genus $g$ over $U$, hence a morphism $U\to\cM_g$ whose image is $\cS_{g,n}$. The image of an irreducible variety under a morphism is irreducible.

For the dimension, the number of branch points of $\mathbf C_{g,n}$ is $r=m$ if $n\mid m$ and $r=m+1$ otherwise, so $\delta=r-3$ gives \cref{eq:dimS}.
\end{proof}

For $n=2$ the locus $\cS_{g,2}$ is the hyperelliptic locus, of dimension $2g-1$.

\begin{prop}\label{prop:Sgn}
Let $g\ge 2$.
\begin{enumerate}[(i)]
\item $\cS_{g,n}=\emptyset$ for $n>2g+2$.
\item For $g+2<n\le 2g+2$, the locus $\cS_{g,n}$ is empty unless $n\in\{2g+1,2g+2\}$, and then it consists of the single point $y^n=x^2-1$, which lies in the hyperelliptic locus.
\item Let $3\le n\le g+2$ and $\cS_{g,n}\ne\emptyset$. Then the datum $\mathbf C_{g,n}$ has $m\ge 3$. If $3\mid g+2$, then $\cS_{g,g+2}$ contains the non-hyperelliptic curve $y^{g+2}=x^3-x$.
\end{enumerate}
\end{prop}

\begin{proof}
By \cref{lem:dichotomy}, a separable model $y^n=h(x)$ of a curve of genus $g\ge2$ has either $m=2$ and $n\in\{2g+1,2g+2\}$, or $m\ge 3$ and $n\le g+2$. This gives (i) and the emptiness statement in (ii). If $m=2$, then $h=c(x-a)(x-b)$, and an affine change of $x$ and a scaling of $y$ bring the model to $y^n=x^2-1$, so the locus is one point; it is hyperelliptic by \cref{lem:dichotomy}. Its dimension is $r-3=0$, since $r=3$. Part (iii) is \cref{lem:dichotomy} together with \cref{thm:bound} and \cref{ex:top}.
\end{proof}

The loci for distinct levels intersect: the Picard curve of \cref{ex:picard} lies in $\cS_{3,3}\cap\cS_{3,4}$, and \cref{ex:picard-alg} recovers both levels from a plane model in which neither is visible. For every $n\in\Lev(C)$ admitting a separable model one has $[C]\in\cS_{g,n}$, and \cref{eq:incl} places $C$ in the strata of all subgroups of the corresponding cyclic groups.

The hypothesis $\cS_{g,n}\ne\emptyset$ in \cref{prop:Sgn} is not vacuous, and it can fail for every level $3\le n\le g+2$ at once.

\begin{rem}\label{rem:genus5}
Let $g=5$. The solutions of \cref{eq:genus} with $m\ge 2$ are
\[
(n,m)\in\{(2,11),\,(2,12),\,(11,2),\,(12,2)\},
\]
so $\cS_{5,n}=\emptyset$ for every $n\notin\{2,11,12\}$. By \cref{prop:Sgn} the loci $\cS_{5,11}$ and $\cS_{5,12}$ are single hyperelliptic points. Hence no non-hyperelliptic curve of genus five satisfies \cref{def:superelliptic}, for any level.

Cyclic covers of genus five with quotient of genus zero are nevertheless abundant. The levels realised are $n\in\{2,3,4,6,8,10,11,12,15,20,22\}$, and for $n\notin\{2,11,12\}$ no branch datum satisfies \cref{eq:sepcond}. Representative data, each of genus five by \cref{eq:RH}, are
\[
\begin{split}
n&=3, \quad (l_j)=(1,1,1,1,1,2,2), \\
n&=4, \quad (l_j)=(1,1,1,1,2,2), \\
n&=6, \quad (l_j)=(1,1,5,5), \\
n&=8, \quad (l_j)=(1,1,2,4).
\end{split}
\]
For $n=3$ the cyclic group is normal in $\Aut(C)$ by \cref{rem:normality}, so these curves are superelliptic in the sense of \cite{beshaj2011, malmendier2019}; for the other levels that sense requires in addition the normality of the cyclic group in $\Aut(C)$, which the branch data alone does not determine. The genus five entries of the tables of \cite{beshaj2011, sanjeewa2008} record the families satisfying it. Since no non-hyperelliptic curve of genus five satisfies \cref{def:superelliptic}, the two conventions share no non-hyperelliptic curve in genus five. Of the levels without a separable model, $n=3$ and $n=4$ satisfy $n\le g$ and every datum above has at least two totally ramified points, so they are recovered by \cref{rem:incomplete}; the levels $6,8,10,15,20,22$ exceed $g$ and lie outside it.

The genus five case is not isolated. For $2\le g\le 100$, equation \cref{eq:genus} has no solution with $n\ge 3$ and $m\ge 3$ exactly for
\[
g\in\{2,\,5,\,8,\,11,\,23,\,29,\,38,\,41,\,53,\,68,\,83,\,89\},
\]
and for these genera $\cS_{g,n}$ is non-empty only for $n=2$ and the two levels of \cref{prop:Sgn}(ii).
\end{rem}

\subsection{Top nodes and the level}\label{sec:top}

Order the loci $\cM(g,G,\mathbf C)$ by inclusion. Every non-trivial $G$ contains a cyclic subgroup of prime order $p$, so by \cref{eq:incl} the maximal elements are among the loci $\cM(g,C_p,\mathbf C)$ with $p$ prime; these are the top nodes of the diagrams of \cite[\S 6.5]{malmendier2019}. To place a curve in a diagram one determines the top nodes containing it, that is the orders of the cyclic subgroups of $\Aut(C)$ with the genus of their quotients, and descends. $\Lev(C)$ indexes the cyclic actions on $C$ whose quotient has genus zero. The top nodes of the diagram arise from subgroups of prime order, but their quotients need not have genus zero, so the level spectrum records only the quotient-genus-zero part of the top-node data, and it is not read off the gonality either. \Cref{alg:main} computes the elements of $\Lev(C)$ admitting a separable model, with the normal form that identifies the stratum. Cyclic subgroups with quotient of positive genus are outside the scope of this paper.

The stratum determines the automorphism group. For a curve in $\cS_{g,n}$ whose cyclic group is normal in $\Aut(C)$, the tables of \cite{sanjeewa2008}, arranged by level in \cite{beshaj2011} and reproduced as \cite[Theorem 21, Table 4]{malmendier2019}, give $\Aut(C)$ together with a parametric equation whose number of parameters is the dimension \cref{eq:delta} of the Hurwitz locus; for $n=3$ and $g\ge 5$ the normality is automatic by \cref{rem:normality}. The normal form returned by \cref{alg:main} is what selects the entry of those tables.

\section{Bounds on the level and Weierstrass points}\label{sec:wp}


Next we get some bounds on the level of each genus $g\geq 2$.   The following is a well known result. 

\begin{lem}\label{lem:dichotomy}
Let $C: y^n=h(x)$ with $h$ separable of degree $m\ge 2$ and $g\ge 2$.
\begin{enumerate}[(i)]
\item If $m=2$, then $n\in\{2g+1,2g+2\}$ and $C$ is hyperelliptic.
\item If $m\ge 3$, then $2g\ge (m-1)(n-2)$ and $n\le g+2$.
\end{enumerate}
\end{lem}

\begin{proof}
Suppose $m=2$. Then \cref{eq:genus} gives $2g=n-\gcd(n,2)$, so $n=2g+2$ if $n$ is even and $n=2g+1$ if $n$ is odd. Moreover $x$ satisfies a quadratic equation over $k(y)$, so $[k(C):k(y)]=2$ and $C$ is hyperelliptic. Suppose $m\ge 3$. Since $\gcd(n,m)\le m$, formula \cref{eq:genus} gives $2g\ge (n-1)(m-1)-m+1=(m-1)(n-2)\ge 2(n-2)$, whence $n\le g+2$.
\end{proof}

\begin{theorem}\label{thm:bound}
Let $C$ admit a cyclic cover of degree $n$ with a model \cref{eq:model} in which $h$ is separable, and let $g\ge 2$. Then $n\le 2g+2$, and the bound is attained by $y^{2g+2}=x^2-1$. If $C$ is not hyperelliptic, then $n\le g+2$, and the bound is attained by $y^n=x^3-x$ with $3\mid n$, $n\ge 6$.
\end{theorem}

\begin{proof}
The two bounds are the two cases of \cref{lem:dichotomy}, since $m\ge 2$ always and $m=2$ forces $C$ hyperelliptic. For $h=x^2-1$ and $n=2g+2$ formula \cref{eq:genus} gives $2g'=(2g+1)-2+1=2g$, so the curve has genus $g$ and level $2g+2$. For $h=x^3-x$ and $3\mid n$ formula \cref{eq:genus} gives $2g'=2(n-1)-3+1=2n-4$, so $g'=n-2$ and $n=g'+2$. This curve is not hyperelliptic by \cref{ex:top}.
\end{proof}

\begin{rem}\label{rem:mssv}
Wiman \cite{wiman1895} showed that an automorphism of a curve of genus $g\ge 2$ has order at most $4g+2$; see \cite{mssv2001} for a survey and \cite{homma1980} for the extremal curves. Since a level is the order of an automorphism, $n\le 4g+2$ for every cyclic cover. \Cref{thm:bound} is sharper by a factor of two, and the difference is the separability hypothesis. Without it the bound $2g+2$ fails, by \cref{eq:RH}:
\[
\begin{array}{lll}
n=10, & (l_j)=(1,4,5), & g=2,\\
n=9,  & (l_j)=(1,2,6), & g=3,\\
n=4g+2, & y^{4g+2}=x^{2g}(x-1)^{2g+1}, & \text{genus } g,
\end{array}
\]
the last attaining Wiman's bound.
\end{rem}

\begin{rem}\label{rem:lemma14}
In \cite[Lemma 14]{malmendier2019} and \cite[Lemma 17]{shor} the inequality $g\ge n$, with equality only for $(n,m)\in\{(2,5),(2,6),(3,4)\}$, is proved for $h$ separable of degree $m$ under the standing hypothesis $m>n$ of \cite[Definition 15]{shor}, through the step that $n\ge 4$ forces $m\ge 5$. Outside that hypothesis the inequality fails:
\[
y^5=x(x^2-1),\ g=4;\qquad y^7=x^3+1,\ g=6;\qquad y^{11}=x^2+1,\ g=5,
\]
the first of which appears in the genus-four classification of \cite{malmendier2019}. \Cref{lem:dichotomy} gives the bounds valid for every $m\ge 2$.
\end{rem}

\begin{exa}\label{ex:top-alg}
Let $C: y^5=x^3-x$, of genus four. By \cref{ex:top}, with $P_\infty$ the point above $x=\infty$, totally ramified since $\gcd(5,3)=1$,
\[
\begin{array}{lll}
\text{point} & \text{sequence} & w\\
\hline
(0,0),\ (\pm 1,0) & (0,1,2,5)=s_{5,3} & 2\\
P_\infty & (0,1,3,6) & 4
\end{array}
\]
This is the case $n=g+1$ excluded by the hypothesis $m>n$ of \cite{shor}.

The curve has a second level. The automorphism $\rho(x,y)=(-x,-\zeta_5 y)$ has order ten; its fixed points are $(0,0)$ and $P_\infty$, the points $(\pm1,0)$ form an orbit with stabiliser of order five, and no other point has a non-trivial stabiliser. Riemann--Hurwitz gives
\[
6=10(2h-2)+9+9+8,
\]
so $C/\langle\rho\rangle$ has genus zero and $10\in\Lev(C)$. Since $10=2g+2>g+2$ and $C$ is not hyperelliptic, \cref{thm:bound} shows that this level has no separable model. Indeed, from $y^{10}=x^2(x^2-1)^2$ one reads off
\[
v^{10}=u\,(u-1)^2,\qquad (u,v)=(x^2,\ y),
\]
with branch data $(1,2,7)$, which fails \cref{eq:sepcond}. So $[C]\in\cS_{4,5}$, and $C$ admits a cyclic cover of degree $10$ without lying in $\cS_{4,10}$. That locus is not empty: here $10=2g+2$, so by \cref{prop:Sgn} it is the single point $y^{10}=x^2-1$, which is hyperelliptic.
\end{exa}

\subsection{Total ramification}\label{sec:total}

Let $C$ be given by \cref{eq:model} and let $p_i$ be a totally ramified branch point with unique point $P$ above it. Then
\begin{equation}\label{eq:div}
\Div(x-p_i)=nP-D_\infty,
\end{equation}
where $D_\infty$ is the polar divisor of $x$, of degree $n$ and supported over $x=\infty$. Hence $1/(x-p_i)$ lies in $L(nP)$ and is not constant, so
\begin{equation}\label{eq:key}
h^0(nP)\ge 2.
\end{equation}
If $p_j$ is a second totally ramified branch point with point $P'$ above it, then
\begin{equation}\label{eq:pair}
\Div\Bigl(\frac{x-p_i}{x-p_j}\Bigr)=nP-nP',
\end{equation}
so $nP\sim nP'$ and $(x-p_i)/(x-p_j)$ spans $L(nP'-nP)$, a space of dimension one.

\begin{prop}\label{prop:small}
Let $C$ admit a cyclic cover of degree $n$, let $g\ge 2$, and let $P$ lie over a totally ramified branch point. If $n\le g$, then $P$ is a Weierstrass point of $C$.
\end{prop}

\begin{proof}
By \cref{eq:key} the integer $n$ is a non-gap at $P$. A point that is not a Weierstrass point has non-gaps $\{0\}\cup\{k: k\ge g+1\}$, forcing $n\ge g+1$.
\end{proof}

\Cref{prop:small} needs no hypothesis beyond total ramification, but it leaves the range $n\ge g+1$ open. That range is not empty: for $m=3$ formula \cref{eq:genus} gives $n\in\{g+1,g+2\}$. Under separability the range is closed by the following theorem.

\subsection{Branch points are Weierstrass points}\label{sec:branchwp}

\begin{theorem}\label{thm:wp}
Let $C: y^n=h(x)$ with $h$ separable of degree $m\ge 2$ and $g\ge 2$, and let $P$ lie above a root of $h$. Then $P$ is a Weierstrass point of $C$ if and only if $m\ge 3$. If $m=2$, the vanishing sequence at $P$ is $(0,1,\dots,g-1)$.
\end{theorem}

\begin{proof}
Suppose $m=2$. By \cref{lem:dichotomy} we have $n\ge 2g+1>2g-2$, so in \cref{eq:snm} only $a=0$ occurs, and $2b\le 2g-2$ gives $b\le g-1$. Hence $s_{n,2}=(0,1,\dots,g-1)$ and $P$ is not a Weierstrass point.

Suppose $m\ge 3$ and, for contradiction, that $P$ is not a Weierstrass point. By \cref{prop:small} we have $n\ge g+1$, hence $g\le n-1$. Together with $2g\ge (m-1)(n-2)$ from \cref{lem:dichotomy} this gives
\[
(m-1)(n-2)\le 2g\le 2(n-1),
\]
so $m-1\le 2+2/(n-2)$ for $n\ge 3$. For $n\ge 5$ the right side is less than three, so $m=3$. For $n=4$ it gives $m\le 4$, and \cref{eq:genus} leaves $(n,m)\in\{(4,3),(4,4)\}$, both of genus three. For $n=3$, formula \cref{eq:genus} gives $g=1$ if $m=3$ and $g\ge 3>n-1$ if $m\ge 4$, so $n=3$ does not occur. In each remaining case $n\le 2g-2$: for $m=3$ we have $g=n-1$ if $3\nmid n$ and $g=n-2$ if $3\mid n$, so $2g-2$ is $2n-4\ge n$ for $n\ge 4$, respectively $2n-6\ge n$ for $n\ge 6$, and $3\mid n$ forces $n\ge 6$; for $(n,m)=(4,4)$ we have $2g-2=4=n$. Hence $(a,b)=(1,0)$ satisfies $an+bm\le 2g-2$, and by \cref{eq:snm} the integer $n$ occurs in the vanishing sequence at $P$. But $n\ge g+1>g-1$, which contradicts the assumption that the sequence is $(0,1,\dots,g-1)$.
\end{proof}

\begin{cor}\label{cor:finite}
Let $C$ have genus $g\ge 2$ and let $\tau$ realise a level $n\ge 2$ of $C$ by a cover with a separable model $y^n=h(x)$, $\deg h=m$. If $m\ge 3$, then $C$ has two Weierstrass points $P\ne P'$ with $nP\sim nP'$, both with vanishing sequence $s_{n,m}$, and $n\le g+2$. If $m=2$, then $C$ is hyperelliptic and $n\in\{2g+1,2g+2\}$. If $C$ is not hyperelliptic, only the first case occurs.
\end{cor}

\begin{proof}
If $m\ge 3$, take $P,P'$ above two roots of $h$. They are Weierstrass points by \cref{thm:wp}, their vanishing sequence is $s_{n,m}$ by \cref{prop:filter}, $nP\sim nP'$ by \cref{eq:pair}, and $n\le g+2$ by \cref{lem:dichotomy}. If $m=2$, apply \cref{lem:dichotomy}.
\end{proof}

The Weierstrass property of the branch points is due to Lewittes \cite{lewittes1963} for $m\ge 5$, and  \cite[Corollary 5.3]{abdon2019} for all $m\ge 3$; \cite[Lemma 5.2]{abdon2019} is \cref{prop:small}. In \cite[Proposition 14]{shor} it is stated under the hypothesis $m>n$, which enters through $g\ge n$; by \cref{rem:lemma14} the inequality does not persist outside it. The proof above replaces $g\ge n$ by \cref{lem:dichotomy}, which holds for all $m\ge 2$, gives the equivalence with $m\ge 3$, and identifies $m=2$ as a genuine exception rather than a case excluded by hypothesis.

For $m=2$ the level satisfies $n\ge 2g+1$, so the non-gap $n$ supplied by \cref{eq:key} lies beyond the last gap and carries no information. The two degree two maps involved are distinct. On $y^n=x^2-1$ the map of degree two is $y$, since $x^2=y^n+1$, and the Weierstrass points of $C$ are the points where $x$ vanishes, together with the point at infinity when $n$ is odd. The branch points of the level $n$ cover are the two points where $y$ vanishes, and the fibre of $y$ over $0$ is unramified. So these points carry the generic vanishing sequence by \cref{thm:wp}, and the level is invisible to a search over Weierstrass points. It is recognised from the hyperelliptic branch locus in step 5 of \cref{alg:main}.
 
\begin{rem}\label{rem:normality}
Let $C$ admit a cyclic cover of degree three, with $g\ge 5$. Then the $g^1_3$ of $C$ is unique, since by the Castelnuovo--Severi inequality \cite{accola1994} two distinct $g^1_3$ force $g\le (3-1)^2=4$. Every automorphism of $C$ then preserves the pencil and normalises its Galois group, so $C_3\trianglelefteq\Aut(C)$. The normality hypothesis of \cite{sanjeewa2008} therefore holds without being assumed for $n=3$ and $g\ge 5$.
\end{rem}
 
\section{The superelliptic normal form}\label{sec:wnf}

Let $t: C\to\PP^1$ be a morphism of degree $n$ such that $k(C)/k(t)$ is Galois with cyclic Galois group. We construct a birational map from $C$ to a curve $v^n=h(u)$ with $h\in k[u]$ and $k(u)=k(t)$. The construction is Kummer theory followed by a reduction of exponents modulo $n$.

\subsection{The Kummer step}\label{sec:kummer}

Let $\tau$ generate $\Gal(k(C)/k(t))$. Since $k$ contains $\zeta_n$, the automorphism $\tau^*$ of the $k(t)$-vector space $k(C)$ is semisimple and
\begin{equation}\label{eq:eigen}
k(C)=\bigoplus_{j=0}^{n-1}W_j,\qquad W_j=\{\,w\in k(C): \tau^*w=\zeta_n^{\,j}w\,\},
\end{equation}
with $W_0=k(t)$.

\begin{lem}\label{lem:kummer}
Each $W_j$ is a one dimensional $k(t)$-vector space. If $s\in W_1\setminus\{0\}$, then $s^n\in k(t)$ and $k(C)=k(t)(s)$.
\end{lem}

\begin{proof}
By the normal basis theorem, $k(C)$ is isomorphic to the group algebra $k(t)[\langle\tau\rangle]$ as a module over it. Since $k(t)$ contains $\zeta_n$, the group algebra is the direct sum of the $n$ character spaces, each of dimension one over $k(t)$; hence $\dim_{k(t)}W_j=1$ for every $j$. Let $s\in W_1\setminus\{0\}$. Then $\tau^*(s^n)=s^n$, so $s^n$ lies in the fixed field $k(t)$. The powers $s^j$ lie in $W_j$ and are non-zero, so they span $\bigoplus_j W_j$ over $k(t)$.
\end{proof}

The proof uses only that the base field contains $\zeta_n$; the lemma therefore holds with $k$ replaced by any field of characteristic zero containing $\zeta_n$ over which $C$, $t$ and $\tau$ are defined.

Write $s^n=\tilde h(t)\in k(t)$ and factor
\begin{equation}\label{eq:htilde}
\tilde h(t)=c\prod_{j=1}^{r}(t-a_j)^{e_j},\qquad c\in k^\times,\ e_j\in\Z\setminus\{0\}.
\end{equation}
The exponents may be negative, so $\tilde h$ need not be a polynomial. The exponents are determined modulo $n$ by the local monodromy, as follows. For a point $R$ of $C$ over a place $a$ of $\PP^1$, of ramification index $e$, the stabiliser of $R$ in $\langle\tau\rangle$ has a distinguished generator, the element acting on a uniformiser $t_R$ by multiplication by $\zeta_n^{n/e}$ modulo $t_R^2$; writing it as $\tau^{l}$, we call $\tau^{l}$ the local monodromy at $a$.

\begin{lem}\label{lem:local}
Let $s\in W_1\setminus\{0\}$ with $s^n=\tilde h(t)$, let $a$ be a place of $\PP^1$ and let $\tau^{l}$ be the local monodromy at $a$. Then $\ord_a(\tilde h)\equiv l\pmod n$. In particular $\ord_a(\tilde h)\equiv 0\pmod n$ at every place that is not a branch point.
\end{lem}

\begin{proof}
Let $R$ be a point over $a$, $v=\ord_R$, $e$ the ramification index and $\sigma=\tau^{l}$. Let $z_a$ be a local parameter at $a$, so $z_a=t-a$ for $a$ finite and $z_\infty=1/t$; then $v(z_a)=e$, and $s^n=\tilde h(t)$ gives $v(s)=(e/n)\ord_a(\tilde h)$. The automorphism $\sigma$ fixes $R$, acts trivially on the residue field $k$, and multiplies a uniformiser by $\zeta_n^{n/e}$ modulo the maximal ideal, so $\sigma(f)/f\equiv\zeta_n^{(n/e)v(f)}$ modulo the maximal ideal for every $f\in k(C)^\times$. For $f=s$ the quotient is the constant $\sigma(s)/s=\zeta_n^{l}$, whence $l\equiv(n/e)v(s)=\ord_a(\tilde h)\pmod n$. If $a$ is not a branch point, the stabiliser is trivial, $l\equiv 0$, and the congruence follows.
\end{proof}

For the model \cref{eq:model} a uniformiser at a point over $p_j$ is $y^\alpha(x-p_j)^\beta$ with $\alpha l_j/d_j+\beta e_j=1$, and $\tau^{l_j}$ multiplies it by $\zeta_n^{\alpha l_j}=\zeta_n^{d_j}$, so the local monodromy at $p_j$ is $\tau^{l_j}$, in agreement with \cref{sec:branch}.

\subsection{Reduction of the exponents}\label{sec:reduce}

Replacing $s$ by $s\,r(t)$ with $r\in k(t)^\times$ replaces $\tilde h$ by $\tilde h\,r^n$ and leaves \cref{eq:eigen} unchanged. The exponents therefore matter only modulo $n$.

\begin{lem}\label{lem:reduce}
Let $\tilde h$ be as in \cref{eq:htilde} and put
\[
r(t)=\prod_{j=1}^{r}(t-a_j)^{-\lfloor e_j/n\rfloor},\qquad u=t,\qquad v=s\,r(t).
\]
Then $v^n=h(u)$ with
\begin{equation}\label{eq:hred}
h(u)=c\prod_{j=1}^{r}(u-a_j)^{\bar e_j},\qquad \bar e_j\equiv e_j\pmod n,\quad 0\le\bar e_j<n,
\end{equation}
so $h\in k[u]$, and $k(u,v)=k(C)$.
\end{lem}

\begin{proof}
By \cref{lem:kummer}, $v^n=\tilde h(t)\,r(t)^n$, and the exponent of $t-a_j$ in $\tilde h\,r^n$ is $e_j-n\lfloor e_j/n\rfloor=\bar e_j\in\{0,\dots,n-1\}$. Since $r(t)\in k(t)^\times$, we have $k(u,v)=k(t)(s)=k(C)$.
\end{proof}

\begin{rem}\label{rem:alt}
The substitution $u=1/(t-a)$, $v=s\,(t-a)^{\deg q}$, with $a$ a root of the denominator $q$ of $\tilde h$, does not produce a polynomial in general. For $\tilde h(t)=(7-t^{11})/t^2$ and $n=3$ it gives $v^3=7t^4-t^{15}$, which in $u=1/t$ is $7u^{-4}-u^{-15}$. \Cref{lem:reduce} gives $r(t)=t$ and $v^3=7u-u^{12}$.
\end{rem}

\subsection{When $h$ is separable}\label{sec:separable}

The branch data of $t$ was defined in \cref{sec:branch} for the model \cref{eq:model}; by \cref{lem:equiv} it is an invariant of the pair $(C,\langle\tau\rangle)$ up to the equivalence stated there.

\begin{theorem}\label{thm:sep}
Let $t: C\to\PP^1$ be cyclic of degree $n$ with branch data $(l_1,\dots,l_r)$, the point at infinity included when it is a branch point. Then there is a birational map from $C$ to a curve $v^n=h(u)$ with $k(u)=k(t)$ and $h\in k[u]$ separable if and only if there is $c\in(\Z/n\Z)^\times$ such that
\begin{equation}\label{eq:sepcond}
c\,l_j\equiv 1\pmod n\quad\text{for every $j$ with at most one exception.}
\end{equation}
\end{theorem}

\begin{proof}
Suppose $v^n=h(u)$ with $h$ separable of degree $m$ and $k(u)=k(t)$. For the generator $\tau: (u,v)\mapsto(u,\zeta_n v)$ the local monodromy at each root of $h$ is $1$, and the only other possible branch point is $u=\infty$, with local monodromy $-m$. So the branch data is $(1,\dots,1,-m)$ for this generator, and \cref{eq:sepcond} holds with $c=1$ if $\tau$ is the given generator, and with the unit relating the two generators otherwise, by \cref{lem:equiv}. Conversely suppose \cref{eq:sepcond} holds for $c$. Replacing $\tau$ by a suitable power multiplies all $l_j$ by $c$, so we may assume $l_j=1$ for all $j$ except possibly one index $j_0$. Apply a M\"obius transformation of $u=t$ carrying $p_{j_0}$ to infinity, or carrying a non-branch point to infinity if there is no exception. By \cref{lem:local} the exponent of $\tilde h=s^n$ at every finite branch point is congruent to $1$ modulo $n$ and its exponent at every other finite place is divisible by $n$, so \cref{lem:reduce} gives $\bar e_j=1$ at every finite branch point and $\bar e_j=0$ elsewhere, and $h$ has simple roots.
\end{proof}

\begin{exa}\label{ex:nonsep}
Let $C: y^4=x^2(x^3-1)$, of genus four by \cref{eq:RH}. The projection $x$ is cyclic of degree four with branch data
\[
(l_0,l_1,l_{\zeta_3},l_{\zeta_3^2},l_\infty)=(2,1,1,1,3).
\]
The units modulo four are $1$ and $3$; for $c=1$ two entries differ from $1$, and for $c=3$ the data is $(2,3,3,3,1)$ with four. So \cref{eq:sepcond} fails, and the projection $x$ admits no model $v^4=h(u)$ with $h$ separable, so this cover realises no separable level. By \cref{def:superelliptic} the curve is nevertheless superelliptic, of level six. \Cref{alg:main} returns
\[
v^6=u(u^2+1),
\]
so $C\cong\{y^6=x^3-x\}$ and $[C]\in\cS_{4,6}$. Indeed $\sigma(x,y)=(-x,\zeta_{12}y)$ is an automorphism of $y^6=x^3-x$ of order twelve; $\sigma^3$ has order four, fixes $(0,0)$ and the three points over $x=\infty$, and has one further orbit of size two with stabiliser of order two, so by \cref{eq:RH} its quotient has genus zero, with branch data $(2,1,1,1,3)$ up to \cref{lem:equiv}.
\end{exa}

\section{Deciding superellipticity}\label{sec:alg}

The input is an absolutely irreducible polynomial $F\in k_0[x,y]$ defining a plane model of $C$ over a subfield $k_0\subseteq k$, with $C$ the smooth projective model of the curve $F=0$; in practice $k_0$ is a number field. No smoothness of the plane model is assumed. The question is geometric: whether $[C]\in\cS_{g,n}$, that is whether $C_{\bar k_0}$ admits a model $y^n=h(x)$ with $h$ separable. This section answers it over an explicit finite extension of $k_0$ generated by the relevant Weierstrass points and roots of unity. \Cref{sec:arith} treats the question of what can be done over $k_0$ itself.

Every object used is a Riemann--Roch space of a divisor rational over the field of computation, computed by the algorithms of Hess \cite{hess2002}, as in van Hoeij \cite{vanhoeij1995,vanhoeij2002}. Points of $C$ enter through Galois orbits over $k_0$, which we call closed points; the residue field of a closed point is the field generated by any of its points.

Throughout this section $\pi:C\to\PP^1$ is a cyclic cover of degree $n$ with a separable model $y^n=h(x)$, $\deg h=m\ge 3$. We write
\[
\begin{split}
S &= \text{the divisor of the $m$ finite branch points},\\
d &= \gcd(n,m),\qquad e=n/d,\\
\bar Q &= \text{the reduced fibre of $x$ over $\infty$, of degree $d$},\qquad F_\infty=e\bar Q=x^*(\infty).
\end{split}
\]
%

\subsection{Weierstrass points and candidates}\label{sec:prelim}

The genus, an integral basis and a basis $\omega_1,\dots,\omega_g$ of $H^0(C,\Omega^1)$ are computed over $k_0$  in \cite{hess2002}. Let $\pi_0$ be a separating variable, $\omega_i=f_i\,d\pi_0$, and let $W=\det\bigl(D^{(j)}f_i\bigr)_{0\le j<g}$ be the Wronskian with respect to the Hasse derivatives $D^{(j)}=(1/j!)\,d^j/d\pi_0^j$. By \cite{towse2000} the divisor
\begin{equation}\label{eq:wdiv}
E=\Div(W)+\tfrac{g(g+1)}{2}\Div(d\pi_0)
\end{equation}
is effective, its support is the set of Weierstrass points, and its multiplicity at $P$ is the weight
\[
w(P)=\sum_{i}(s_i-i)
\]
of the vanishing sequence $(s_i)$ at $P$; in characteristic zero $\deg E=g^3-g$. Since $W\in k_0(C)$, the divisor $E$ is $k_0$-rational, and its closed points lie above the places of $k_0(\pi_0)$ occurring in the norm $N(W)=N_{k_0(C)/k_0(\pi_0)}(W)$ or in $\Div(d\pi_0)$. The height of $W$ is not bounded in terms of $g$, so $W$ is used through its norm wherever possible. The next proposition supplies the part of $E$ above a place of $k_0(\pi_0)$ from that norm, and a criterion discarding a place without computing the points above it. Let $w_{\min}$ be the least weight $w(s_{n,m})$ over the pairs $(n,m)$ with $2\le n\le g+2$ satisfying \cref{eq:genus} for the given $g$. These are the pairs with $m\ge 3$, so $w_{\min}\ge 1$ by \cref{thm:wp}, and they are the pairs entering steps 2 to 4 of \cref{alg:main}; the excluded pairs $(2g+1,2)$ and $(2g+2,2)$ have generic sequence, of weight zero, and their levels are decided in step 5.

\begin{prop}\label{prop:wnorm}
Let $q$ be a place of $k_0(\pi_0)$, let $e_q=\ord_q\bigl(N(W)\bigr)$, and let $D_q$ be the degree of the part of $\Div(d\pi_0)$ above $q$. Then
\[
e_q = \sum_{P\mid q}f(P\mid q)\,v_P(W),
\]
and 
\[
\sum_{P\mid q}\deg P\cdot w(P) = e_q\deg q+\tfrac{g(g+1)}{2}D_q.
\]
If
\[
e_q\deg q+\tfrac{g(g+1)}{2}D_q<w_{\min}\deg q,
\]
 then no point of $E$ above $q$ has weight $w(s_{n,m})$ for a pair $(n,m)$ with $n\le g+2$ satisfying \cref{eq:genus}.
\end{prop}

\begin{proof}
The first identity is the factorisation of the norm into local norms. For the second, $\deg P=f(P\mid q)\deg q$ for every $P\mid q$, and $w(P)=v_P(W)+\tfrac{g(g+1)}{2}v_P(d\pi_0)$ by \cref{eq:wdiv}; summing $\deg P\cdot w(P)$ over $P\mid q$ gives $\deg q\sum_{P\mid q}f(P\mid q)v_P(W)+\tfrac{g(g+1)}{2}\sum_{P\mid q}\deg P\cdot v_P(d\pi_0)$, which is the right hand side. For the last statement, $E$ is effective, so every term of the sum is non-negative, and a point $P\mid q$ of weight $w(s_{n,m})\ge w_{\min}$ contributes $\deg P\cdot w(P)\ge w_{\min}\deg q$.
\end{proof}

The weight of a closed point of degree one is computed from its vanishing sequence, through the dimensions $\ell(K-iP)$, without reference to $W$. When exactly one point of degree greater than one lies above $q$, the first identity of \cref{prop:wnorm} determines $v_P(W)$ at that point, hence its weight, after the points of degree one above $q$ have been subtracted; only when two or more such points share a place is a valuation of $W$ computed directly. A closed point entering a candidate branch divisor has degree at most $m$, and $m\le 2g+2$ by \cref{eq:genus}, since $2g\ge m-2$ for $n\ge3$ and $m=2g+\gcd(2,m)$ for $n=2$; places of larger degree are discarded as well.

For $2\le n\le g+2$ let $\Sigma_n(g)$ be the finite set of sequences \cref{eq:sn} at the totally ramified points of the branch data of level $n$ and genus $g$, and $w(\Sigma_n(g))$ the set of their weights. A point whose weight is not in $w(\Sigma_n(g))$ is not a totally ramified branch point of a level $n$ cover; this broader filter serves the search for non-separable levels of \cref{rem:incomplete}, while the candidates \cref{eq:cand} of \cref{alg:main} use only the weight $w(s_{n,m})$ below. For a level with a separable model the sequence at a branch point is $s_{n,m}$ of \cref{eq:snm}, and by \cref{cor:finite} all $m$ finite branch points are Weierstrass points. The candidate branch divisors of level $n$ are therefore
\begin{equation}\label{eq:cand}
\begin{split}
\cC_n=\Bigl\{\,S=P_1+\dots+P_m:\ & P_i \text{ distinct Weierstrass points of } C_{\bar k_0},\ w(P_i)=w(s_{n,m}),\\
& (n,m) \text{ satisfies \cref{eq:genus}}\Bigr\}.
\end{split}
\end{equation}
This is a finite set. Let $k_n$ be the compositum of $k_0(\zeta_n)$ with the splitting fields over $k_0$ of the closed points of $E$ of weight $w(s_{n,m})$, $(n,m)$ satisfying \cref{eq:genus}; over $k_n$ every Weierstrass point of these weights is rational. Every $S\in\cC_n$ is a $k_n$-rational divisor whose points are $k_n$-rational.

\begin{exa}\label{ex:filter}
Let $g=3$. The branch data with a totally ramified point, the sequences \cref{eq:sn} at those points, and their weights are the following; for $n=5$ formula \cref{eq:RH} has no solution.
\[
\begin{array}{cllll}
n & \text{data} & \text{point} & \text{sequence} & w\\
\hline
3 & (1,1,1,1,2) & l_j=1 & s_{3,4}=(0,1,3) & 1\\
  &             & l_j=2 & (0,1,4) & 2\\
4 & (1,1,1,1)   & \text{all} & s_{4,4}=s_{4,3}=(0,1,4) & 2\\
4 & (1,1,3,3)   & \text{all} & (0,1,2) & 0\\
4 & (1,1,2,2,2) & l_j=1 & (0,2,4) & 3
\end{array}
\]
So $\Sigma_3(3)=\{(0,1,3),(0,1,4)\}$ and $\Sigma_4(3)=\{(0,1,4),(0,1,2),(0,2,4)\}$. Step 2 keeps the closed points of $E$ of multiplicity $w(s_{3,4})=1$ for $n=3$ and of multiplicity $w(s_{4,4})=w(s_{4,3})=2$ for $n=4$; the remaining weights $2$, $0$ and $3$ of the table belong to non-separable data and enter only the search of \cref{rem:incomplete}. The generic sequence has weight zero and is carried by no closed point of $E$. The data $(1,1,3,3)$ is that of \cref{rem:incomplete}. The data $(1,1,2,2,2)$ gives a hyperelliptic curve, since $\tau^2$ has quotient of genus zero, and its totally ramified points are hyperelliptic Weierstrass points; this level has $n=g+1$ and non-separable data, so it is covered neither by \cref{thm:complete} nor by \cref{rem:incomplete}.
\end{exa}

\subsection{The eigenspaces from the branch divisor}\label{sec:eigen}

Let $\tau(x,y)=(x,\zeta_ny)$ and, for $0\le b'\le n-1$,
\[
V_{b'}=\{\,\omega\in H^0(C,\Omega^1):\ \tau^*\omega=\zeta_n^{-(n-1-b')}\omega\,\}.
\]
By \cref{prop:basis}, $V_{b'}$ has basis $x^a\,dx/y^{\,n-1-b'}$ with $an+b'm\le 2g-2$. Put
\begin{equation}\label{eq:bmaxA}
b_{\max}=\Bigl\lfloor\frac{2g-2}{m}\Bigr\rfloor,\qquad A=\Bigl\lfloor\frac{2g-2-b_{\max}m}{n}\Bigr\rfloor,
\end{equation}
so that $\dim V_{b_{\max}}=A+1$ and $V_{b'}=0$ for $b'>b_{\max}$.

\begin{prop}\label{prop:eigen}
For $0\le i\le b_{\max}$,
\[
H^0\bigl(C,\Omega^1(-iS)\bigr)=\bigoplus_{b'\ge i}V_{b'}.
\]
In particular, for $b_{\max}\ge 1$, $V_{b_{\max}}$ and $V_{b_{\max}-1}\oplus V_{b_{\max}}$ are the Riemann--Roch spaces of the divisors $K-b_{\max}S$ and $K-(b_{\max}-1)S$, over any field over which $S$ is rational.
\end{prop}

\begin{proof}
For $\varphi\in k[x]$ with $\deg\varphi\le(2g-2-b'm)/n$ the differential $\varphi(x)\,dx/y^{n-1-b'}$ has order $n\,\ord_{p_j}(\varphi)+b'$ at the point above $p_j$, by the proof of \cref{prop:basis}. Let $\omega=\sum_{b'}\varphi_{b'}(x)\,dx/y^{n-1-b'}\in H^0(C,\Omega^1)$. The orders of the summands at a fixed branch point are pairwise incongruent modulo $n$, so the order of $\omega$ there is their minimum. Hence $\omega$ vanishes to order at least $i$ at every branch point if and only if $h\mid\varphi_{b'}$ for every $b'<i$. Since $\deg\varphi_{b'}\le(2g-2)/n<m$ by \cref{eq:genus}, this forces $\varphi_{b'}=0$.
\end{proof}

The proposition gives a necessary condition on a candidate pair $(S,n)$,
\begin{equation}\label{eq:dimtest}
h^0(K-iS)=\sum_{b'=i}^{b_{\max}}\Bigl(\Bigl\lfloor\frac{2g-2-b'm}{n}\Bigr\rfloor+1\Bigr),\qquad 0\le i\le b_{\max},
\end{equation}
which we call the dimension test. It also gives the divisors of the two distinguished differentials, of which $\omega_0$ is defined for every $b_{\max}\ge 0$ and $\omega_1$ only for $b_{\max}\ge 1$,
\begin{equation}\label{eq:omega01}
\omega_0=\frac{dx}{y^{\,n-1-b_{\max}}},\qquad \omega_1=\frac{dx}{y^{\,n-b_{\max}}}=\frac{\omega_0}{y}.
\end{equation}
From $\Div(y)=S-\tfrac md\bar Q$ and $\Div(dx)=(n-1)S-(e+1)\bar Q$,
\begin{equation}\label{eq:divomega}
\begin{split}
\Div(\omega_0)&=b_{\max}S+c\,\bar Q,\\
\Div(\omega_1)&=(b_{\max}-1)S+\bigl(c+\tfrac md\bigr)\bar Q,
\end{split}
\qquad c=\frac{(n-1-b_{\max})m}{d}-e-1.
\end{equation}
Here $c\ge 0$ is an integer, and $c=0$ if and only if $m'\mid n'+1$, where $m=dm'$, $n=dn'$; this uses $n-1-b_{\max}=\lceil(n+d)/m\rceil$, which follows from $2g-2=(n-1)m-n-d$. If $b_{\max}=0$, then $V_0=H^0(C,\Omega^1)$, so $A=g-1\ge 1$ and only case (a) of \cref{thm:quotient} below occurs; $\omega_1$ is not used.


\begin{lem}\label{lem:derivative}
Let $D$ be one of the divisors $F_\infty$ and $nQ$, $Q$ a point of $S$, and put $S'=S$, respectively $S'=S-Q$. Then
\[
\{\,f\in L(D):\ \ord_P(df)\ge n-1 \text{ for every } P\in S'\,\}=
\begin{cases}
\langle 1,\ x\rangle, & D=F_\infty,\\
\langle 1,\ 1/(x-x(Q))\rangle, & D=nQ.
\end{cases}
\]
\end{lem}

\begin{proof}
Write $f=\sum_{j=0}^{n-1}f_j(x)y^j$ with $f_j\in k(x)$. At a point $P$ above a root $p$ of $h$ we have $\ord_P(y)=1$, $\ord_P(dx)=n-1$ and $\ord_P(dy)=0$, so
\[
\ord_P\,d(f_jy^j)=j-1+n\,\ord_p(f_j)\quad(j\ge 1),\qquad \ord_P\,d(f_0)\ge n-1.
\]
These orders are pairwise incongruent modulo $n$, so $\ord_P(df)\ge n-1$ forces $\ord_p(f_j)\ge 1$ for every $j\ge 1$ and every $p$ with $P\in S'$. For $D=F_\infty$ the pole bound gives $\deg f_j\le 1-jm/n<1$, so $f_j$ is a constant for $j\ge1$, hence zero, and $f_0$ has degree at most one. For $D=nQ$, $x(Q)=q$, the pole bound at $Q$ and regularity elsewhere give $f_j=c_j/(x-q)$ or $f_j\in k$ for $j\ge1$, and $f_j$ vanishes at the $m-1\ge 2$ roots of $h$ other than $q$; so $f_j=0$, and $f_0\in\langle 1,1/(x-q)\rangle$.
\end{proof}

The condition of the lemma is linear over every field over which $D$ and $S'$ are rational, which is $k_n$ in this section and $k_0$, or the stated residue field, in \cref{sec:arith}: for a basis $f_1,\dots,f_r$ of $L(D)$ one asks for which $(c_i)$
\[
\sum_i c_i\,df_i\in H^0\bigl(\Omega^1(-(n-1)S'+2D)\bigr),
\]
a Riemann--Roch space of a divisor rational over that field.

\begin{theorem}\label{thm:quotient-geom}
For $S\in\cC_n$ and a point $Q$ of $S$, the subspace
\[
U(S,Q)=\{\,f\in L(nQ):\ \ord_P(df)\ge n-1 \text{ for every } P\in S-Q\,\}
\]
is defined over $k_n$. If $S$ is the divisor of the finite branch points of $\pi$, then $U(S,Q)$ has dimension two and the ratio of a basis is a M\"obius transform of the quotient function $x$.
\end{theorem}

\begin{proof}
Over any field over which $Q$ and $S-Q$ are rational, and $k_n$ is such a field for $S\in\cC_n$, the space $L(nQ)$ has a basis and the condition is a system of linear equations, namely membership of $\sum c_i\,df_i$ in $H^0(\Omega^1(-(n-1)(S-Q)+2nQ))$; its solution space is defined over that field. When $S$ is the branch divisor, \cref{lem:derivative} identifies the solution space as $\langle 1,1/(x-x(Q))\rangle$ over $\bar k_0$, of dimension two, so it contains a non-constant $k_n$-rational element $u=(\alpha+\beta x)/(\gamma+\delta x)$.
\end{proof}

The theorem uses only the last case of \cref{lem:derivative}. The other constructions of the lemma, and the case analysis of \cref{sec:arith}, are needed only when one insists on working over $k_0$.


Let $u\in k(C)$ have degree $n$, with finite branch values $a_1,\dots,a_r\in k$. Suppose that
\begin{equation}\label{eq:fibres}
\begin{split}
u^*(a_j)&=e_j\,F_j,\qquad \deg F_j=d_j,\qquad e_jd_j=n\quad(1\le j\le r),\\
u^*(\infty)&=e_\infty F_\infty,\qquad \deg F_\infty=d_\infty,\qquad e_\infty d_\infty=n,
\end{split}
\end{equation}
with $F_j$, $F_\infty$ reduced; otherwise $u$ is not Galois. Suppose further that $\gcd(n,d_1,\dots,d_r,d_\infty)=1$; otherwise $u$ is not cyclic, since the local monodromies of a cyclic cover generate its Galois group.

\begin{prop}\label{prop:cert}
Under \cref{eq:fibres}, the following are equivalent.
\begin{enumerate}[(i)]
\item $k(C)/k(u)$ is Galois with cyclic Galois group.
\item There are integers $l_1,\dots,l_r$ with $\gcd(n,l_j)=d_j$ for all $j$ and $\gcd(n,l_1+\dots+l_r)=d_\infty$ such that
\begin{equation}\label{eq:Ddiv}
D=\frac1n\Div\Bigl(\prod_{j=1}^{r}(u-a_j)^{l_j}\Bigr)
\end{equation}
is principal.
\end{enumerate}
When (ii) holds and $v$ spans $L(-D)$, there is $c\in k^\times$ with
\[
v^n=c\prod_{j=1}^r(u-a_j)^{l_j},\qquad k(C)=k(u,v).
\]
\end{prop}

\begin{proof}
The divisor $D$ is integral: at a point over $a_j$ the order of $u-a_j$ is $e_j$ and $l_je_j/n=l_j/d_j\in\Z$; at a point over $\infty$ the order of $\prod_j(u-a_j)^{l_j}$ is $-e_\infty\sum_j l_j$ and $d_\infty\mid\sum_j l_j$; elsewhere the order is zero.

Suppose (i), with generator $\tau$, and let $s\in W_1$ for this $\tau$, so that $s^n=\tilde h(u)\in k(u)$ by \cref{lem:kummer}. Let $\tau^{l_j}$ be the local monodromy at $a_j$ and $\tau^{l_\infty}$ that at $\infty$; then $\gcd(n,l_j)=d_j$ and $\gcd(n,l_\infty)=d_\infty$, since the stabiliser $\langle\tau^{l_j}\rangle$ has order $e_j$. By \cref{lem:local}, $l_\infty\equiv\ord_\infty(\tilde h)\equiv-\sum_j\ord_{a_j}(\tilde h)\equiv-\sum_j l_j\pmod n$, so $\gcd(n,\sum_j l_j)=d_\infty$. By \cref{lem:local} the order of $\tilde h$ at $a_j$ is congruent to $l_j$ modulo $n$ and its order at every other finite place is divisible by $n$. Hence every exponent of the rational function $\prod_j(u-a_j)^{l_j}/\tilde h$ at a finite place is divisible by $n$, so its order at infinity is too, and it equals $c'q(u)^n$ with $q\in k(u)^\times$ and $c'\in k^\times$. Then $D-\Div(s)=\tfrac1n\Div\bigl(\prod_j(u-a_j)^{l_j}/\tilde h\bigr)=\Div(q)$, and $D$ is principal.

Suppose (ii) and let $\Div(v)=D$. Then $v^n/\prod_j(u-a_j)^{l_j}$ has trivial divisor, hence equals a constant $c\in k^\times$. The polynomial $X^n-c\prod_j(u-a_j)^{l_j}$ is irreducible over $k(u)$ by Capelli's theorem: a prime $\ell$ dividing $n$ and every $l_j$ would divide $d_1,\dots,d_r$ and $d_\infty$, contradicting $\gcd(n,d_1,\dots,d_r,d_\infty)=1$, and $-4$ is a fourth power in $k$; see \cite[Ch.~VI, Thm.~9.1]{lang2002}. So $[k(u,v):k(u)]=n=[k(C):k(u)]$, whence $k(C)=k(u,v)$ is a Kummer extension of $k(u)$, Galois with cyclic group.
\end{proof}

The choices in (ii) are finite: $l_j$ ranges over $d_j$ times a unit modulo $e_j$. Over $\C$ they may be read off the monodromy of $u$, computed by numerical analytic continuation as in \cite{deconinck2001}; in general they are enumerated, with the $l_j$ constant on the Galois orbits of the branch values over the field of computation, so that no field extension is needed. For the genuine candidates required in \cref{thm:complete} and \cref{thm:complete-k0} this restriction loses nothing. Over $k_n$, the candidate the algorithm must accept has $k_n$-rational totally ramified branch points, and $\zeta_n\in k_n$; the generator of the deck group whose induced action on the cotangent space at one of them is multiplication by $\zeta_n$ is therefore Galois-invariant, hence defined over $k_n$, and \cref{lem:kummer}, applied over $k_n(u)$, gives a Kummer equation over $k_n$, whose exponents are Galois-stable and reduce to the $(l_j)$ by \cref{lem:local}. Over $k_0$ this is \cref{rem:cert-k0}. The test that $D$ is principal is exact in either case. When (ii) holds, the pair $(u,v)$ with $0<l_j<n$ is the normal form of \cref{lem:reduce}, with $h(u)=c\prod_j(u-a_j)^{l_j}\in k[u]$, and $h$ is separable if and only if $(l_j)$ satisfies \cref{eq:sepcond}, by \cref{thm:sep}.

The branch values $a_j$ of a candidate $u$ are read off a plane model in which $u$ is one of the coordinates. The next lemma certifies such a model and confines the branch values to the roots of one polynomial.

\begin{lem}\label{lem:branchvalues}
Let $L=k_0(C)$, let $u\in L$ have degree $n$, and let $z\in L$ satisfy $\gcd(\deg z,n)=1$. Then $k_0(u,z)=L$. If moreover $G\in k_0[T,Z]$ is irreducible with $G(u,z)=0$ and $\deg_ZG=n$, then every finite branch value of $u$ is a root of $\disc_Z(G)\cdot\operatorname{lc}_Z(G)$.
\end{lem}

\begin{proof}
The degree $[L:k_0(u,z)]$ divides $[L:k_0(u)]=n$ and $[L:k_0(z)]=\deg z$, which are coprime, so it equals one. Let $a\in k_0$ satisfy $\operatorname{lc}_Z(G)(a)\ne0$ and $\disc_Z(G)(a)\ne0$. Then $G(a,Z)$ has degree $n$ and $n$ distinct roots. Every place of $L$ above $u=a$ has $z$ finite, since $\operatorname{lc}_Z(G)(a)\ne0$, and its $z$-value is a root of $G(a,Z)$; distinct places above $u=a$ have distinct $z$-values, since $k_0(u,z)=L$. Hence there are at most $n$ such places, and $\sum_{P\mid a}e_Pf_P=n$ forces exactly $n$ places with $e_P=f_P=1$. So $a$ is not a branch value.
\end{proof}


The levels with $m=2$ are invisible to Weierstrass points by \cref{thm:wp}, and are decided from the hyperelliptic model instead. Let $K\subseteq k$ be a field, let $C:y_2^2=f(x_2)$ with $f\in K[x_2]$ separable of degree $2g+1$ or $2g+2$, and let $t\in\PP^1(K)$. Choose $\gamma_t\in\mathrm{PGL}_2(K)$ with $\gamma_t(t)=\infty$, and let $f_t\in K[x]$ be the polynomial for which $y^2=f_t(x)$ is the model obtained from $y_2^2=f(x_2)$ by $x=\gamma_t(x_2)$. For $t$ finite and $\gamma_t(x_2)=1/(x_2-t)$,
\[
f_t(x)=x^{2g+2}f(t+1/x)=\sum_{i=0}^{2g+2}\frac{f^{(i)}(t)}{i!}\,x^{2g+2-i},
\]
and $f_\infty=f$. Another choice of $\gamma_t$ replaces $f_t$ by $c^2f_t(\lambda x+\mu)$ with $c,\lambda\in K^\times$ and $\mu\in K$. Since $f$ is separable, $f(t)$ and $f'(t)$ do not both vanish, so $\deg f_t\in\{2g+1,2g+2\}$, with $\deg f_t=2g+2$ if and only if $f(t)\ne0$.

\begin{prop}\label{prop:shape}
Let $g\ge2$, let $C:y_2^2=f(x_2)$ with $f\in K[x_2]$ separable of degree $2g+1$ or $2g+2$, let $n\in\{2g+1,2g+2\}$, and put $m=n-1$. Let $T_n$ be the set of $t\in\PP^1(K)$ for which
\[
f_t=a\,(x-\beta)^n+b,\qquad a,b\in K^\times,\ \beta\in K .
\]
\begin{enumerate}[(i)]
\item There are $u,v\in K(C)$ with $K(u,v)=K(C)$ and $v^n=(u^2-b)/a$ for some $a,b\in K^\times$ if and only if $T_n\ne\emptyset$.
\item $T_n$ is finite. It contains $\infty$ only if $\deg f=n$, and a finite $t\in T_n$ is a root of $f$ if $n=2g+1$ and satisfies $f(t)\ne0$ if $n=2g+2$.
\item Let $n=2g+2$ and let $t\in K$. Then $t\in T_n$ if and only if $p_t=f_t'$ satisfies
\begin{equation}\label{eq:shape}
m\,p_t\,p_t''=(m-1)\,(p_t')^2 .
\end{equation}
The coefficients of $m\,p_t\,p_t''-(m-1)(p_t')^2$ as a polynomial in $x$ are polynomials in $t$ over $K$, not all zero, and $T_n\cap K$ is the set of their common roots in $K$.
\end{enumerate}
\end{prop}

\begin{proof}
(i) Let $t\in T_n$ and put $u=y$, $v=x-\beta$ in the model $y^2=f_t(x)=a(x-\beta)^n+b$. Then $u^2=av^n+b$, so $v^n=(u^2-b)/a$, and $K(u,v)=K(x,y)=K(C)$. Conversely let $v^n=(u^2-b)/a$ with $K(u,v)=K(C)$, that is $u^2=av^n+b$. The polynomial $av^n+b$ is separable of degree $n\ge5$, so this is a hyperelliptic model of $C$ over $K$ with $x$-coordinate $v$. The hyperelliptic map is unique for $g\ge2$, so $K(v)=K(x_2)$ and $v=M(x_2)$ with $M\in\mathrm{PGL}_2(K)$. Put $t=M^{-1}(\infty)\in\PP^1(K)$. Then $M\gamma_t^{-1}$ fixes $\infty$, hence is affine over $K$, and $f_t$ is the transform of $av^n+b$ under it, up to a square factor in $K^\times$; so $f_t=a'(x-\beta')^n+b'$ with $a',b'\in K^\times$ and $t\in T_n$.

(ii) For $t\in T_n$ the model $y^2=f_t(x)$ has $\deg f_t=n$, so $t=\infty$ requires $\deg f=n$, and a finite $t$ satisfies $f(t)=0$ when $n=2g+1$ and $f(t)\ne0$ when $n=2g+2$. For finiteness, let $t\in T_n$ and write $f_t=a(x-\beta)^n+b$; the roots of $f_t$ are $\beta+\rho\,\mu_n$ with $\rho^n=-b/a$, so $\delta_t=\gamma_t$ followed by $x\mapsto(x-\beta)/\rho$ carries the branch locus of $C$ to the set $\mu_n$ of $n$-th roots of unity, together with $\infty$ when $n$ is odd. Distinct $t$ give distinct $\delta_t$, since $t=\delta_t^{-1}(\infty)$. Two such maps differ by an element of $\mathrm{PGL}_2(k)$ stabilising that target set, which has at least three points, and the stabiliser of a finite set of at least three points is finite. Hence $T_n$ is finite.

(iii) Put $L=p_t'/p_t$. Dividing \cref{eq:shape} by $p_t^2$ and using $p_t''/p_t=L'+L^2$ gives $mL'+L^2=0$. Write $p_t=c\prod_i(x-\beta_i)^{\mu_i}$ with the $\beta_i$ distinct, so that $L=\sum_i\mu_i/(x-\beta_i)$ and $L'=-\sum_i\mu_i/(x-\beta_i)^2$. Comparing the coefficients of the double pole at $\beta_i$ in $L^2=-mL'$ gives $\mu_i^2=m\mu_i$, so $\mu_i=m$ for every $i$. Now $\deg f_t\le 2g+2$ gives $\deg p_t\le 2g+1=m$, and $\deg p_t\ge1$ since $\deg f_t\ge 2g+1\ge5$; so $p_t$ has exactly one root, of multiplicity $m$, and $p_t=a'(x-\beta)^{m}$ with $a'$ the leading coefficient of $p_t$. Both lie in $K$: the coefficient of $x^{m-1}$ in $p_t$ is $-a'm\beta$ and $a'\in K^\times$. Integrating, $f_t=a(x-\beta)^n+b$ with $a=a'/n\in K^\times$ and $b=f_t(\beta)\in K$, and $b\ne0$ because $f_t$ is separable. Conversely $p_t=na(x-\beta)^{m}$ satisfies \cref{eq:shape}. The coefficients of $f_t$ are the $f^{(i)}(t)/i!$, hence polynomials in $t$, and so are the coefficients of $m\,p_t\,p_t''-(m-1)(p_t')^2$. If all of them vanished identically, then every $t\in K$ would lie in $T_n$, contradicting the finiteness in (ii), since $K$ is infinite in characteristic zero.
\end{proof}

\begin{algorithm}\label{alg:main}
Input: $F\in k_0[x,y]$ absolutely irreducible, of genus $g\ge2$. Output: the set of levels $n$ of $C_{\bar k_0}$ admitting a separable model, each with $(u,v,h)$ such that $v^n=h(u)$, $h\in k_n[u]$ separable, and $\bar k_0(u,v)=\bar k_0(C)$.
\begin{enumerate}[1.]
\item Compute $g$, a basis of $H^0(C,\Omega^1)$, and the Weierstrass divisor $E$ of \cref{eq:wdiv} over $k_0$.
\item For each $n$ with $2\le n\le g+2$ form $k_n$ and the candidate set $\cC_n$ of \cref{eq:cand} over $k_n$, and keep the $S\in\cC_n$ passing the dimension test \cref{eq:dimtest}.
\item For each surviving $(S,n)$ and a point $Q$ of $S$ compute the subspace $U(S,Q)$ of \cref{thm:quotient-geom} over $k_n$. If $\dim U(S,Q)\ne 2$, discard the candidate; otherwise let $u$ be a non-constant ratio of a basis of $U(S,Q)$.
\item Apply \cref{prop:cert} to $u$ over $k_n$. If $u$ is not cyclic, discard the candidate. If $u$ is cyclic with local data $(l_j)$ and \cref{eq:sepcond} fails, discard the candidate; if \cref{eq:sepcond} holds, output $n$ with $(u,v,h)$, where $h$ is made separable by \cref{lem:reduce} and \cref{thm:sep}.
\item If the level $2$ was output with model $y_2^2=f(x_2)$, then for each $n\in\{2g+1,2g+2\}$ form the finite set $T_n$ of \cref{prop:shape}: it contains $\infty$ when $\deg f=n$, the roots of $f$ in the field of computation when $n=2g+1$, and the common roots in that field of the coefficients of $m\,p_t\,p_t''-(m-1)(p_t')^2$, $m=n-1$, when $n=2g+2$. For each $t\in T_n$ write $f_t=a(x-\beta)^n+b$ and output $n$ with
\[
v^n=\frac{u^2-b}{a},\qquad (u,v)=(y,\ x-\beta),
\]
in the coordinates of the model $y^2=f_t(x)$.
\end{enumerate}
\end{algorithm}

Every output is verified by division: writing $v^n-h(u)$ as a rational function of $x$ and $y$,
\[
F \mid \operatorname{num}\bigl(v^n-h(u)\bigr),\qquad \deg u^*(\infty)=n .
\]
By \cref{prop:cert}, \cref{thm:sep} and the division check the algorithm has no false positives.

Step 5 is the only place where hyperelliptic curves are treated separately, and only for the two levels with $m=2$. By \cref{lem:dichotomy} a separable model of level $2g+1$ or $2g+2$ has $m=2$, hence after completing the square reads $y^n=x^2-c$ with $c\in k^\times$, which is the model of \cref{prop:shape}(i); by \cref{thm:wp} its branch points are not Weierstrass points, so these two levels are invisible to steps 2 to 4, and each is realised by a single curve (\cref{prop:Sgn}). Correctness of step 5 is \cref{prop:shape}, applied over the field of computation: part (i) is the equivalence between the level and the shape of $f_t$, and parts (ii) and (iii) make the set of candidates $t$ finite and computable without any enumeration of M\"obius transformations.

For $n\in\{2g+1,2g+2\}$ put $k_n=k_2(\zeta_{2n})$. If $[C]\in\cS_{g,n}$, then $C$ is hyperelliptic and, by the model $y^n=x^2-c$, its branch locus is carried to the roots of $x^n+1$, together with $\infty$ when $n$ is odd, by a M\"obius transformation $\gamma$. The branch points of the level $2$ model are rational over $k_2$ and those roots over $k_0(\zeta_{2n})$, and $\gamma$ is determined by the images of three of them, so $\gamma\in\mathrm{PGL}_2(k_n)$ and $t=\gamma^{-1}(\infty)\in\PP^1(k_n)$. Hence $T_n\ne\emptyset$ over $k_n$, and step 5 run over $k_n$ outputs $n$ with a model and a transformation defined over $k_n$.

\begin{theorem}\label{thm:complete}
Let $C$ be given by $F\in k_0[x,y]$, of genus $g\ge 2$, and let $n\ge 2$. Then $[C]\in\cS_{g,n}$ if and only if \cref{alg:main} outputs $n$. When it does, the model $v^n=h(u)$ returned is defined over $k_n$, and so is the birational transformation.
\end{theorem}

\begin{proof}
If the algorithm outputs $n$, then $u$ is cyclic of degree $n$ by \cref{prop:cert} and $h$ is separable, so $[C]\in\cS_{g,n}$. Conversely let $\pi:C_{\bar k_0}\to\PP^1$ realise a level $n$ with a separable model $y^n=h(x)$, $\deg h=m$. If $m=2$, then $C$ is hyperelliptic by \cref{lem:dichotomy} and $n\in\{2g+1,2g+2\}$; the level $2$ has $m'\ge5$ finite branch points, is treated by the rest of the argument, and $n$ is then found in step 5, with model and transformation defined over $k_n=k_2(\zeta_{2n})$. Let $m\ge 3$. By \cref{cor:finite} the finite branch points are Weierstrass points with sequence $s_{n,m}$, hence of weight $w(s_{n,m})$, so $S\in\cC_n$; and $S$ passes the dimension test by \cref{prop:eigen}. By \cref{thm:quotient-geom} the subspace of step 3 has dimension two and yields a M\"obius transform $u$ of $x$ over $k_n$; its branch data is $(1,\dots,1)$, together with $-m$ at infinity when $n\nmid m$, so $u$ is cyclic, step 4 accepts it by \cref{prop:cert}, condition \cref{eq:sepcond} holds with $c=1$, and \cref{thm:sep} makes $h$ separable. All computations are Riemann--Roch spaces of $k_n$-rational divisors, so the output is defined over $k_n$.
\end{proof}

The field $k_n$ can be large. For $2\le n\le g+2$ its degree over $k_0$ is finite, at most $\varphi(n)\prod_P(\deg P)!$ over the closed points $P$ of $E$ whose weight is $w(s_{n,m})$ for some pair $(n,m)$ satisfying \cref{eq:genus}; for $n\in\{2g+1,2g+2\}$ it is at most $\varphi(2n)\,[k_2:k_0]$. It is a convenient uniform field over which every candidate splits pointwise; it need not be minimal, and an individual level and its branch points may be defined over much smaller fields. \Cref{sec:arith} determines when the algorithm can be run over $k_0$ or over a residue field of a single closed point of $S$, which is the situation in every example of this paper.


\begin{rem}\label{rem:incomplete}
Levels without a separable model are outside \cref{thm:complete}. When such a level is realised by a cover with two totally ramified points $P,P'$ and $n\le g$, the points are Weierstrass points by \cref{prop:small} and $L(nP'-nP)$ is spanned by a cyclic function; this recovers the level after a base change making $P,P'$ rational. 
It can fail beyond that range: the branch data $n=4$, $(l_j)=(1,1,3,3)$ has genus three, every branch point is totally ramified, and by \cref{eq:sn} the vanishing sequence at each is the generic $(0,1,2)$, so no branch point is a Weierstrass point.
Whole genera can consist of such levels. By \cref{rem:genus5} every level of a non-hyperelliptic curve of genus five is without a separable model, and six of the eight such levels have $n>g$.
\end{rem}

\section{Descent to the field of definition}\label{sec:arith}

The extension $k_n$ of \cref{thm:complete} was chosen so that every candidate branch divisor splits pointwise; the branch points of an individual level generate a smaller field in general. Over a number field one wants the model $y^n=h(x)$ over $k_0$ itself, or over a controlled extension that carries it. This section determines what \cref{alg:main} decides over $k_0$, and at what cost in field extensions.

Throughout, $\pi:C\to\PP^1$ is a cyclic cover of degree $n$ with a separable model $y^n=h(x)$, $\deg h=m\ge3$, and $\pi$ and the model are defined over $k_0$; the notation $S$, $d$, $e$, $\bar Q$, $F_\infty$, $b_{\max}$, $A$, $c$, $\omega_0$, $\omega_1$ is that of \cref{sec:alg}. Both $S$ and $\bar Q$ are then $k_0$-rational divisors, and \cref{prop:eigen}, \cref{eq:dimtest}, \cref{eq:omega01} and \cref{eq:divomega} hold over $k_0$.

\begin{rem}\label{rem:descent}
The assumption that $\pi$ is defined over $k_0$ has two parts.
First, $\langle\tau\rangle\subset\Aut(C)$ must be Galois-stable; by \cite{hidalgo} the superelliptic group of a given level is unique for $n$ odd and, up to explicit exceptions, for $n$ even, and a unique group is stable.
 Second, $C/\langle\tau\rangle$ is then a $k_0$-form of $\PP^1$, that is a conic, which is $\PP^1_{k_0}$ if and only if it has a $k_0$-point; over a number field it holds after at most one quadratic extension. When the level $n$ group is not unique, its conjugates give distinct covers with the same branch data, each defined over the field fixing it.
\end{rem}

\subsection{Candidates}\label{sec:cand-k0}

A branch divisor of a cover defined over $k_0$ is a $k_0$-rational divisor, hence a sum of closed points of $E$. The candidates of level $n$ over $k_0$ are
\begin{equation}\label{eq:cand-k0}
\begin{split}
\cC_n(k_0)=\Bigl\{\,S=\textstyle\sum_i P_i:\ & P_i \text{ distinct closed points of } E,\ w(P_i)=w(s_{n,m}),\\
& \textstyle\sum_i\deg P_i=m,\ (n,m)\text{ satisfies \cref{eq:genus}}\Bigr\},
\end{split}
\end{equation}
a subset of the $\cC_n$ of \cref{eq:cand} formed without extending $k_0$. In practice it is small: ordinary Weierstrass points have weight one, and the only cases with $w(s_{n,m})=1$ are $(n,m)=(3,4)$ in genus three and $(n,m)\in\{(2,5),(2,6)\}$ in genus two.

\subsection{The quotient function}\label{sec:quotient-k0}

Over $k_n$ the quotient function is obtained from a single rational branch point by \cref{thm:quotient-geom}. Over $k_0$ no branch point need be rational, and the other constructions of \cref{lem:derivative}, together with the eigenspaces of \cref{prop:eigen}, replace it.

\begin{theorem}\label{thm:quotient}
Let $S$ be the branch divisor of $\pi$, and let $b_{\max}$, $A$ be as in \cref{eq:bmaxA} and $c$ as in \cref{eq:divomega}. A M\"obius transform of $x$ is obtained over $k_0$ from Riemann--Roch spaces of $k_0$-rational divisors, after finitely many candidate constructions each tested in step 4, in each of the following cases:
\begin{enumerate}[(a)]
\item $A\ge 1$;
\item $A=0$ and $c\ge 1$;
\item $A=0$, $c=0$ and $d=1$;
\item $A=0$, $c=0$, $d>1$, and either $d<m<n$ or $S$ contains a $k_0$-rational point.
\end{enumerate}
In the remaining case,
\begin{enumerate}[(a)]
\item[(d$'$)] $A=0$, $c=0$, $d>1$, $S$ has no $k_0$-rational point, and $m=d$ or $m\ge n$,
\end{enumerate}
the same holds after one base change to the residue field of a closed point of $S$, of degree at most $m$. The five cases are exhaustive: $c=0$ gives $m\mid n+d$, so $m=d$, or $d<m<n$, or $m\in\{n,n+d\}$, and $m=n$ forces $m=d$. Applied to an arbitrary candidate $S$, failure of an asserted divisor shape or dimension below causes the candidate to be discarded; the proof guarantees success for the genuine branch divisor.
\end{theorem}

\begin{proof}
(a) $V_{b_{\max}}=\{\varphi(x)\omega_0:\deg\varphi\le A\}$, and $\omega_0$ has no zeros outside $S\cup\bar Q$. Imposing vanishing at a closed point $T\notin S$ cuts this space to a subspace of the same shape
\[
\{\,q\,\psi(x)\,\omega_0:\ \deg\psi\le A'\,\}:
\]
for $T$ with $x(T)$ finite the condition is $q_T\mid\varphi$, with $q_T$ the minimal polynomial of $x(T)$ over $k_0$, and for $T$ over $x=\infty$ it is a bound on $\deg\varphi$. The algorithm enumerates the closed points $T\notin S$ by increasing degree and height of the places of $k_0(\pi_0)$ below them, an effective enumeration over a number field, and imposes vanishing at $T$ whenever the dimension drops by exactly one, until it is two; the drop is read off the Riemann--Roch space, and no knowledge of $x$ is used. At a finite point the drop equals $\deg q_T$ or zero, so a drop by one forces $\deg x(T)=1$; every fibre of $x$ over a value in $\PP^1(k_0)$ outside the finitely many values already forced contains such a point, and $\PP^1(k_0)$ is infinite, so the enumeration terminates. A space of dimension two of this shape is $\{q\,\psi\,\omega_0:\deg\psi\le 1\}$, and the ratio of a basis is $(\alpha+\beta x)/(\gamma+\delta x)$.

(b) $V_{b_{\max}}=\langle\omega_0\rangle$, and \cref{eq:divomega} gives
\[
\bar Q=\frac{\Div\omega_0-b_{\max}S}{c},\qquad F_\infty=e\bar Q,
\]
after which \cref{lem:derivative} with $D=F_\infty$ returns $\langle 1,x\rangle$. For an arbitrary candidate the right-hand side of the display need not be effective with multiplicities divisible by $c$ and of degree $d$; if it is not, the candidate is discarded.

(c) With $d=1$ the fibre over $\infty$ is one point $Q_\infty$ with $\ord(x)=-n$, so $F_\infty=nQ_\infty$. The point $Q_\infty$ is $k_0$-rational, and it is a Weierstrass point not in $S$: the differential $dx/y^{n-1}$ is holomorphic with divisor $((n-1)m-n-1)\,Q_\infty=(2g-2)\,Q_\infty$, and $2g-2>g-1$. For each $k_0$-rational Weierstrass point $Q\notin S$ form the subspace of \cref{lem:derivative} for $D=nQ$, $S'=S$, and retain it only if it has dimension two. For $Q=Q_\infty$ one has $nQ_\infty=F_\infty$, the lemma applies, and the subspace is $\langle 1,x\rangle$; for any other $Q$ the lemma makes no claim, and a retained subspace is tested and, if wrong, rejected in step 4.

(d) If $S$ contains a $k_0$-rational point $Q$, then \cref{lem:derivative} with $D=nQ$, $S'=S-Q$ applies over $k_0$, for every $m$. Suppose $d<m<n$. Here $b_{\max}\ge 1$, since $A=0$, and $c=0$ gives $2g-2=b_{\max}m$, so by \cref{prop:eigen} the space $H^0(K-(b_{\max}-1)S)=V_{b_{\max}-1}\oplus V_{b_{\max}}$ has dimension $\lfloor m/n\rfloor+2=2$ and equals $\langle\omega_0,\omega_1\rangle$; the ratio $y'$ of a $k_0$-basis is a M\"obius transform of $y$, so the fibres of $y'$ are the fibres of $y$. By \cref{eq:divomega} with $c=0$ the members of the pencil have divisors $(b_{\max}-1)S+Z$ with $Z$ a fibre of $y'$; the fibre through $\bar Q$ is $\tfrac md\bar Q$, non-reduced since $d<m$, over a value in $\PP^1(k_0)$ since $y'$ is $k_0$-rational and constant on the $k_0$-rational divisor $\bar Q$. The fibre $\tfrac md\bar Q$ need not be the only non-reduced one: $y'$ ramifies at the points above the roots of $h'$, so its fibres over the critical values in $k^\times$ are non-reduced as well. The ramification points of $y'$ are read off the ramification divisor $R_{y'}=\Div(dy')+2\,(y')_\infty$ over $k_0$; evaluating $y'$ at them gives the finitely many critical values, hence the non-reduced fibres. For each value $t\in\PP^1(k_0)$ whose fibre $Z$ satisfies $Z=\tfrac md\,\bar Z$ with $\bar Z$ reduced of degree $d$, form the subspace of \cref{lem:derivative} for $D=e\bar Z$, $S'=S$, and retain it only if it has dimension two. For $\bar Z=\bar Q$ one has $e\bar Z=F_\infty$, the lemma applies, and the subspace is $\langle 1,x\rangle$; for any other choice the lemma makes no claim, and a retained subspace is tested and, if wrong, rejected in step 4.

(d$'$) A base change to the residue field of a closed point of $S$, of degree at most $m$, makes a root $Q$ of $h$ rational, and \cref{lem:derivative} with $D=nQ$, $S'=S-Q$ applies. Here $c=0$ gives $m\mid n+d$, so $m=d$, which divides $n$, or $m=n+d$.
\end{proof}

\subsection{Completeness}\label{sec:complete-k0}

\begin{rem}\label{rem:cert-k0}
The test of \cref{prop:cert} is performed over $k_0$ for the covers of \cref{thm:complete-k0}. For a cover admitting a Kummer equation over $k_0$, the local data may be chosen constant on the Galois orbits of the branch values: by \cref{lem:local} the $l_j$ are the exponents modulo $n$ of $\tilde h\in k_0(u)$, which are Galois-stable; the divisor $D$ of \cref{eq:Ddiv} is then $k_0$-rational. The dimension $h^0(-D)$ does not change under extension of the constant field, so $D$ is principal over $\bar k_0$ if and only if it is principal over $k_0$, and $v$ may be taken in $k_0(C)$, with $h\in k_0[u]$ and, since $F$ is absolutely irreducible, $c\in k_0^\times$. The conclusion of \cref{prop:cert} is geometric: $\bar k_0(C)/\bar k_0(u)$ is cyclic, while $k_0(C)/k_0(u)$ itself is Galois only when $\zeta_n\in k_0$.
\end{rem}

\begin{theorem}\label{thm:complete-k0}
Let $C$ have genus $g\ge 2$ and let $n\in\Lev(C)$ be realised by a cyclic cover $\pi:C\to\PP^1$ with a separable model $y^n=h(x)$, $\deg h\ge 3$, defined over $k_0$. Then \cref{alg:main}, run over $k_0$ with the candidates \cref{eq:cand-k0} and \cref{thm:quotient} in step 3, outputs $n$, and the candidate arising from the branch divisor of $\pi$ yields a normal form of $\pi$. The constant field is extended at most once, to the residue field of a closed point of the branch divisor, hence by degree at most $\deg h$, and only in case (d$'$) of \cref{thm:quotient}, that is only when no root of $h$ is $k_0$-rational and either $\deg h$ divides $n$ or $\deg h=n+\gcd(n,\deg h)$ with $\gcd(n,\deg h)>1$.
\end{theorem}

\begin{proof}
By \cref{cor:finite} the finite branch points are Weierstrass points with sequence $s_{n,m}$, hence of weight $w(s_{n,m})$, and $S$ is a $k_0$-rational sum of closed points of $E$ of that weight and degree $m$; so $S\in\cC_n(k_0)$, and it passes the dimension test by \cref{prop:eigen}. \Cref{thm:quotient} returns a M\"obius transform $u$ of $x$ over $k_0$, or over the stated extension, after finitely many candidates discarded in steps 3 and 4; $u$ is cyclic, step 4 accepts it by \cref{prop:cert}, performed over $k_0$ by \cref{rem:cert-k0}, condition \cref{eq:sepcond} holds with $c=1$, and \cref{thm:sep} makes $h$ separable. Finally, case (d$'$) requires no $k_0$-rational root and $m=d$ or $m\ge n$; the first is $\deg h\mid n$, and the second, under $c=0$, is $\deg h\in\{n,n+d\}$, with $\deg h=n$ contained in $\deg h\mid n$. Conversely $\deg h\mid n$ forces $d=\deg h>1$, and with $d>1$ each of these forces $A=0$ and $c=0$ by \cref{eq:bmaxA} and \cref{eq:divomega}.
\end{proof}

\begin{prop}\label{prop:exc}
Let $C$ have genus $g\ge 2$ and let $n\in\{2g+1,2g+2\}$ be realised by a cyclic cover with a separable model $y^n=h(x)$, $\deg h=2$, defined over $k_0$. Then \cref{alg:main}, run over $k_0$, outputs $n$ in step 5 with a model $v^n=q(u)$, $q\in k_0[u]$ separable of degree two, and a transformation defined over $k_0$; the constant field is not extended.
\end{prop}

\begin{proof}
Write $h=Ax^2+Bx+C$ and put $u=2Ax+B$, $v=y$. Then
\[
v^n=\frac{u^2-b}{a},\qquad a=4A\in k_0^\times,\quad b=B^2-4AC\in k_0^\times,
\]
with $b\ne0$ since $h$ is separable, and $k_0(u,v)=k_0(C)$. The polynomial $av^n+b$ is separable of degree $n\ge5$, so this is a hyperelliptic model of $C$ over $k_0$ and the level $2$ is realised over $k_0$ by a cover with separable branch polynomial of degree $n\ge3$. By \cref{thm:complete-k0} the run over $k_0$ outputs the level $2$, with a model $y_2^2=f(x_2)$ over $k_0$; since $n\ge5$, neither $n\mid2$ nor $n=2+\gcd(2,n)$, so no extension occurs there. The displayed model is that of \cref{prop:shape}(i) over $K=k_0$, so $T_n\ne\emptyset$. By \cref{prop:shape}(ii) and (iii), step 5 computes $T_n$ over $k_0$ and outputs $n$ with $q=u^2-b$ separable of degree two, and with a transformation defined over $k_0$.
\end{proof}

\begin{cor}\label{cor:complete-k0}
Let $C$ have genus $g\ge 2$ and let $n\in\Lev(C)$ be realised by a cyclic cover $\pi$ with a separable model $y^n=h(x)$ defined over $k_0$. Then \cref{alg:main}, run over $k_0$, outputs $n$; for $\deg h\ge 3$ the candidate arising from the branch divisor of $\pi$ yields a normal form of $\pi$, and for $\deg h=2$ step 5 yields one. The constant field is extended at most once, by degree at most $\deg h$, and only when $3\le\deg h$, no root of $h$ is $k_0$-rational, and either $\deg h$ divides $n$ or $\deg h=n+\gcd(n,\deg h)$ with $\gcd(n,\deg h)>1$.
\end{cor}

\begin{proof}
\Cref{thm:complete-k0} for $\deg h\ge 3$ and \cref{prop:exc} for $\deg h=2$.
\end{proof}

Conversely, a level found over $k_0$ by \cref{thm:quotient} and \cref{prop:cert} is a level of $C_{\bar k_0}$; so the run over $k_0$ returns a subset of the geometric answer of \cref{thm:complete}, equal to it exactly when every level with separable model is realised by a cover and model defined over $k_0$. For $n$ odd, \cref{rem:descent} removes the subgroup obstruction and splits the quotient conic after at most one quadratic extension; any remaining obstruction to a Kummer model over the field must be treated separately.

\begin{exa}\label{ex:picard-alg}
Let $F(X,Y)=Y^3-(X+Y)^4+1$, the Picard curve $y^3=x^4-1$ of \cref{ex:picard} in the coordinates $(x,y)=(X+Y,Y)$. Its Weierstrass divisor $E$ over $\Q$ contains the closed points
\[
\begin{array}{lcc}
\text{closed point} & \text{degree} & \text{multiplicity}\\
\hline
(X,Y)=(\pm 1,0);\quad X^2+1=0,\ Y=0 & 1,1,2 & 1\\
(X,Y)=(-\eta,\eta),\ \eta^3=-1 & 1,2 & 2\\
Q_\infty & 1 & 2
\end{array}
\]
and further points of multiplicity one, of degrees two and four. For $(n,m)=(3,4)$ the only $S\in\cC_3$ of weight $w(s_{3,4})=1$ and degree four among these is the first row. Here $b_{\max}=1$, $A=0$, $c=0$, $d=1$, so case (c) of \cref{thm:quotient} applies: $Q_\infty$ is a rational Weierstrass point off $S$, and \cref{lem:derivative} with $D=3Q_\infty$ returns $\langle1,x\rangle$, giving
\[
v^3=u^4-1,\qquad (u,v)=(X+Y,\ Y).
\]
For $(n,m)=(4,3)$ the surviving candidate is the second row; the sum of $Q_\infty$ and the degree two point of the second row also has weight two and degree three, so it enters $\cC_4(k_0)$ as well, and yields no output. Here $b_{\max}=1$, $A=0$, $c=1$, $d=1$, so case (b) applies: $\Div\omega_0=S+\bar Q$ exhibits $\bar Q=Q_\infty$, and \cref{lem:derivative} with $D=F_\infty=4Q_\infty$ gives
\[
v^4=u^3+1,\qquad (u,v)=(Y,\ X+Y).
\]
By \cref{lem:dichotomy} every level of $C$ with a separable model satisfies $n\le 5$, and $n=5$ does not occur by \cref{ex:filter}; so $\{3,4\}$ is the set of levels of $C$ admitting a separable model.
\end{exa}

\begin{rem}\label{rem:cost}
The cost of the algorithm lies in the genus and integral basis over $k_0$, in the determinant of the Wronskian, in the factorisation of its norm, and in Riemann--Roch spaces of divisors of degree at most $2g-2+2n$. By \cref{prop:wnorm} the Wronskian enters through its norm and through valuations read off that norm; a valuation of $W$ itself is computed only when two or more closed points of degree greater than one lie above a single place of $k_0(\pi_0)$, as in \cref{sec:prelim}, and no expansion or residue field of $W$ is formed. This matters in practice: the cost of such operations is governed by the height of $W$, which grows with the height of the input model, while the remaining operations are Riemann--Roch spaces of divisors of degree $O(g)$, resultants, and factorisations over $k_0$. No enumeration over geometric points occurs, and $|\cC_n(k_0)|$ is bounded by the number of subsets of the closed points of $E$ of weight $w(s_{n,m})$, which is small since ordinary Weierstrass points have weight one while the only cases with $w(s_{n,m})=1$ are $(n,m)=(3,4)$ in genus three and $(n,m)\in\{(2,5),(2,6)\}$ in genus two.
\end{rem}

\subsection{Implementation and timings}\label{sec:impl}

An implementation of \cref{alg:main} in Sage accompanies the paper, in the file \texttt{superelliptic.py}, with a README documenting its routines and the bounds a run places on the enumerations of \cref{thm:quotient}. It is the version of this section: it runs over $k_0$ with the candidates \cref{eq:cand-k0} and the cases of \cref{thm:quotient}, and performs the base change of case (d$'$) where the theorem requires it. The function field of the plane model, its genus, a basis of holomorphic differentials, places, divisors and Riemann--Roch spaces are computed by Sage's implementation of the algorithms of Hess \cite{hess2002}. All computations are exact, and every output is verified by \cref{prop:cert} and by division into the input equation.

\Cref{tab:timings} reports \texttt{superelliptic\_levels} on ten plane models over $\Q$, on one core, with Sage 10.4: the genus, the degrees of the closed points of $E$ retained, the levels output, the case of \cref{thm:quotient} used for each of them, the time spent on the determinant of the Wronskian, and the total wall time, both in seconds. The models are linear disguises of normal forms; the one model in the table that is already a normal form, in row two, is recognised in $0.6$ seconds.

\begin{table}[ht]
\centering
\footnotesize
\begin{tabular}{llclcrr}
\hline
$F(X,Y)$ & $g$ & $E$ (degrees) & levels & case & $\det W$ & total \\
\hline
$Y^3-(X+Y)^4+1$ & 3 & $1^4\,2^6\,4$ & 3, 4 & (c), (b) & $0.0$ & $7.3$ \\
$Y^3-X^4-X-1$ & 3 & $1\,4\,18$ & 3 & (c) & $0.0$ & $0.6$ \\
$(X+Y)^4+Y^4-1$ & 3 & $1^4\,2^2$ & 4 & (b) & $0.0$ & $3.1$ \\
$(Y-X)^5-(X+2Y)^3+(X+2Y)$ & 4 & $1^4$ & 5 & (c) & $15.8$ & $28.1$ \\
$(Y+X)^6-X^3+X$ & 4 & $1^4\,12\,36$ & 6 & (d) & $0.7$ & $30.2$ \\
$Y^4-X^2(X^3-1)$ & 4 & $1^2\,2\,12\,36$ & 6 & (d) & $0.2$ & $1.8$ \\
$(Y+X)^2-X^5+1$ & 2 & $1^2\,4$ & 2, 5 & (a), step 5 & $0.0$ & $0.4$ \\
$(Y+X)^2-X^{11}+1$ & 5 & $1^2\,10$ & 2, 11 & (a), step 5 & $0.2$ & $0.8$ \\
$(Y+X)^4-X^5+1$ & 6 & $1^2\,4^2$ & 4, 5 & (c), (b) & $130.8$ & $133.9$ \\
$(Y-2X)^5-X^4+1$ & 6 & $1^4\,2\,4\,20$ & 4, 5 & (c), (b) & $352.6$ & $361.5$ \\
\hline
\end{tabular}
\caption{Runs of \cref{alg:main} over $\Q$.}
\label{tab:timings}
\end{table}

The table bears out \cref{rem:cost}. The determinant of the Wronskian dominates on the two models of genus six, where it takes $130.8$ of $133.9$ seconds and $352.6$ of $361.5$. It is not the leading cost elsewhere: on row five it takes $0.7$ seconds against $12.7$ for the factorisation of the norm together with the weights read off it, and $15.8$ for the level $6$ attempt. \Cref{prop:wnorm} removes the expansions and residue fields of $W$ from the computation, and its valuations except in the case of \cref{sec:prelim}, but not the determinant itself, and it is the determinant that is left.

Row one is \cref{ex:picard-alg} and row six the curve of \cref{ex:nonsep}; rows seven and eight are hyperelliptic, with the levels $5$ and $11$ from step 5. The four cases (a) to (d) of \cref{thm:quotient} all occur in the table; case (d$'$) does not, and no example in this paper reaches it. It is the only case in which the constant field is extended, and by \cref{thm:complete-k0} the extension has degree at most $\deg h$.

Two questions are left open. Levels realised only by cyclic covers that admit no separable model are outside \cref{thm:complete}, and \cref{rem:incomplete} gives branch data, $n=4$ with $(l_j)=(1,1,3,3)$ in genus three, for which the reduction to Weierstrass points fails; deciding those levels needs a different search. Over $k_0$, the run returns the geometric answer of \cref{thm:complete} exactly when every level with a separable model is realised by a cover and a model defined over $k_0$, and \cref{rem:descent} leaves the remaining obstruction to a Kummer model over the field untreated. No complexity bound is claimed: the cost is governed by the height of the Wronskian, which is not bounded in terms of $g$.

\bibliographystyle{amsplain}
\bibliography{sh-92}


\end{document}